\documentclass[10pt]{amsart}
\numberwithin{equation}{section}
\usepackage{xcolor}
\usepackage{mathabx}
\usepackage{lmodern}
\usepackage[T1]{fontenc}
\usepackage{amsmath}
\usepackage[pdftex,textwidth=410pt,marginratio=1:1]{geometry}
\allowdisplaybreaks
\usepackage[english]{babel}
\usepackage{amsfonts}
\usepackage[colorinlistoftodos]{todonotes}
\usepackage{amsmath}
\usepackage{amsthm}
\usepackage{amssymb}
\usepackage{bbm}
\usepackage{cancel}
\usepackage{color}
\usepackage[curve]{xypic}
\usepackage{graphicx}
\usepackage{hyperref}
\usepackage{stmaryrd} %\brackets
\usepackage{comment}
\newtheorem{thm}{Theorem}[section]
\newtheorem{cor}[thm]{Corollary}
\newtheorem{prop}[thm]{Proposition}
\newtheorem{lem}[thm]{Lemma}
\newtheorem{conj}[thm]{Conjecture}

\theoremstyle{definition}
\newtheorem{defi}[thm]{Definition}
\newtheorem{example}[thm]{Example}

\newtheorem{rmk}[thm]{Remark}

\DeclareFontFamily{OT1}{rsfs}{}
\DeclareFontShape{OT1}{rsfs}{n}{it}{<-> rsfs10}{}
\DeclareMathAlphabet{\curly}{OT1}{rsfs}{n}{it}

 \newcommand{\C}{\mathbb{C}}

 \newcommand{\Z}{\mathbb{Z}}
\newcommand{\vir}{\mathrm{vir}}
\newcommand{\mov}{\mathrm{mov}}
 \newcommand{\Q}{\mathbb{Q}}
 \newcommand{\DT}{\mathop{\rm DT}\nolimits}
\newcommand{\PT}{\mathop{\rm PT}\nolimits}
\newcommand{\Exp}{ \mathrm{Exp} }
\newcommand{\TT}{\mathbf{T}}
\newcommand{\fix}{\mathrm{fix}}
\newcommand{\pt}{{\mathsf{pt}}}
\newcommand\mdot{{\scriptscriptstyle\bullet}}
   \DeclareMathOperator{\half}{\mathrm{half}}

\newcommand{\BC}{{\mathbb{C}}}

\newcommand{\BE}{{\mathbb{E}}}
\newcommand{\BF}{{\mathbb{F}}}

\newcommand{\BI}{{\mathbb{I}}}

\newcommand{\BQ}{{\mathbb{Q}}}

\newcommand{\BZ}{{\mathbb{Z}}}

\newcommand{\eE}{\mathcal{E}}
\newcommand{\fF}{\mathcal{F}}

\newcommand{\hH}{\mathcal{H}}
\newcommand{\iI}{\mathcal{I}}

\newcommand{\oO}{\mathcal{O}}

\newcommand{\uU}{\mathcal{U}}

\newcommand{\wW}{\mathcal{W}}
\newcommand{\xX}{\mathcal{X}}
\newcommand{\yY}{\mathcal{Y}}
\newcommand{\zZ}{\mathcal{Z}}

\newcommand{\Hom}{\mathop{\rm Hom}\nolimits}

\newcommand{\Hilb}{\mathop{\rm Hilb}\nolimits}

\newcommand{\ch}{\mathop{\rm ch}\nolimits}
\newcommand{\rk}{\mathop{\rm rk}\nolimits}
\newcommand{\td}{\mathop{\rm td}\nolimits}
\newcommand{\Ext}{\mathop{\rm Ext}\nolimits}

\newcommand{\rank}{\mathop{\rm rank}\nolimits}

\newcommand{\cneq}{\mathrel{\raise.095ex\hbox{:}\mkern-4.2mu=}}
\newcommand{\eqcn}{\mathrel{=\mkern-4.5mu\raise.095ex\hbox{:}}}

\newcommand{\HOM}{\mathop{{\hH}om}\nolimits}

\DeclareRobustCommand{\SkipTocEntry}[4]{}

\title[Crepant resolution conjecture on Calabi-Yau 4-folds]{A Donaldson-Thomas crepant resolution conjecture \\ on Calabi-Yau 4-folds}
\date{}
\author{Yalong Cao}
\address{RIKEN Interdisciplinary Theoretical and Mathematical Sciences Program (iTHEMS), 2-1, Hirosawa, Wako-shi, Saitama, 351-0198, Japan}
\email{yalong.cao@riken.jp}
\author{Martijn Kool}
\address{Mathematical Institute, Utrecht University, P.O.~Box 80010 3508 TA Utrecht, The Netherlands}
\email{m.kool1@uu.nl}
\author{Sergej Monavari}
\address{\'Ecole Polytechnique F\'ed\'erale de Lausanne (EPFL),  CH-1015 Lausanne, Switzerland}
\email{sergej.monavari@epfl.ch}

\begin{document}
\maketitle
\centerline{\emph{Dedicated to Professor Miles Reid on the occasion of his 75th birthday}}
\begin{abstract}
Let $G$ be a finite subgroup of $\mathrm{SU}(4)$ such that its elements have age at most one. In the first part of this paper, we define $K$-theoretic stable pair invariants on a crepant resolution of the affine quotient $\BC^4/G$, and conjecture a closed formula for their generating series in terms of the root system of $G$. In the second part, we define degree zero Donaldson-Thomas invariants of Calabi-Yau 4-orbifolds, develop a vertex formalism that computes the invariants in the toric case, and  conjecture closed formulae for their generating series for the quotient stacks $[\BC^4/\BZ_r]$, $[\BC^4/\BZ_2\times \BZ_2]$. Combining these two parts, we formulate a crepant resolution correspondence which relates the above two theories. 
\end{abstract}
\setcounter{tocdepth}{1}
\tableofcontents

\section{Introduction}

\subsection{Stable pairs on crepant resolutions}

Let $G < \mathrm{SU}(4)$ be a finite group acting on $\BC^4$ by matrix multiplication. In this paper we assume the elements of $G$ have age at most 1. Then, up to unitary conjugation, $G$ is a finite subgroup of $\mathrm{SO}(3) < \mathrm{SU}(3) < \mathrm{SU(4)}$ or $\mathrm{SU}(2) < \mathrm{SU}(4)$ (see Proposition \ref{classify g=1 abelian}). 
Here $\mathrm{SU}(3) < \mathrm{SU}(4)$ is the embedding induced by acting trivially on the fourth coordinate of $\BC^4$ and $\mathrm{SU}(2) < \mathrm{SU}(4)$ is the embedding induced by acting trivially on the third and fourth coordinate of $\BC^4$.
We consider the Nakamura $G$-Hilbert scheme $X := G\textrm{-Hilb}(\BC^4)$ \cite{Nak}. In the case $G < \mathrm{SO}(3)$, we have $X \cong G\textrm{-Hilb}(\BC^3) \times \BC$, where $G\textrm{-Hilb}(\BC^3)$ is irreducible and a crepant resolution of $\BC^3/G$ \cite{BKR}. In the case $G < \mathrm{SU}(2)$, we have $X \cong G\textrm{-Hilb}(\BC^2) \times \BC^2$, where $G\textrm{-Hilb}(\BC^2)$ is irreducible and a minimal crepant resolution of $\BC^2/G$ \cite{IN2}. So in either case we obtain a crepant resolution 
\begin{equation}\label{equ on cr intro}X = Y \times \BC \to \BC^4 / G = \C^3 / G \times \BC. \end{equation} 
By Reid's generalized McKay correspondence \eqref{equ on mckay} and our assumption on age, contracted \emph{proper} subvarieties in \eqref{equ on cr intro} have dimension 1.

Let $P:=P_n(X, \beta)$ be the moduli space of \emph{Pandharipande-Thomas stable pairs} \cite{PT} on $X$ parametrizing pairs $(F,s)$, where $F$ is a pure 1-dimensional sheaf on $X$ with proper support such that $[F]=\beta\in H_2(X, \BZ)$, $\chi(F)=n$, and the cokernel of $s \in H^0(X,F)$ is 0-dimensional. 
Denote by $A_*(P, \BZ\left[\tfrac{1}{2}\right])$ the Chow group of $P$ and by $K_0(P, \BZ\left[\tfrac{1}{2}\right])$ the Grothendieck group of coherent sheaves on $P$ both with $\BZ\left[\tfrac{1}{2}\right]$-coefficients.
By \cite{OT}, one can endow $P$ with a \emph{virtual class} and a \emph{twisted virtual structure sheaf}
\[[P]^{\vir}\in A_n(P, \BZ\left[\tfrac{1}{2}\right]), \quad \widehat{\oO}_P^{\vir}\in K_0(P, \BZ\left[\tfrac{1}{2}\right]), \]
which depend on a choice of orientation on $P$. Existence of orientations was established in \cite{CGJ, Bojko}. 
Stable pair invariants are defined by integrating suitable insertions in cohomology (resp.~$K$-theory) against the virtual class (resp.~twisted virtual structure sheaf).
For a line bundle $L$ on $X$, we consider the \emph{tautological complex}:
\[
L^{[n]}:=\mathbf{R}\pi_{P*}(\pi_X^* L \otimes \mathbb{F}),
\]
where $(\oO\to \BF)$ is the universal stable pair on $X \times P$, and $\pi_X, \pi_{P}$ are the natural projections. 
The Calabi-Yau 4-fold $X = G\textrm{-Hilb}(\BC^4)$ is not proper, but it has an action by an algebraic torus $\TT$ preserving the Calabi-Yau volume form. Moreover, the fixed locus $P^{\TT} \subset P$ is proper (cf.~\S \ref{sect on Nak G Hilb}).
In \cite{CKM}, we defined the following $K$-theoretic invariants
\[P_{n,\beta}(X,L, y):=\chi\left(P, \widehat{\oO}_P^{\vir}\otimes \widehat{\Lambda}^{\mdot} (L^{[n]} \otimes y^{-1})\right)\in K_0^{\TT}(\pt)_{\mathrm{loc}}(y^{\frac{1}{2}}),\]
where $\widehat{\Lambda}^\bullet(\cdot):=\Lambda^{\bullet}(\cdot)\otimes \det(\cdot)^{-1/2}$, $y$ is a formal variable, and $K_0^{\TT}(\pt)_{\mathrm{loc}}$ is the localized $\TT$-equivariant $K$-group of a point \eqref{equ for k0 loc}. Since the moduli space $P$ is not proper, these invariants are defined by equivariant localization (cf.~\S\ref{sec: virtual localization}). 

We form the $K$-theoretic Pandharipande-Thomas (PT) stable pair \emph{partition function} by
\begin{align*}
\mathcal{Z}^{\PT}_{X,L}(y,Q,q)&:=1+\sum_{\beta>0,n\in \mathbb{Z}}P_{n,\beta}(X,L, y)\,Q^{\beta}q^n\in K_0^{\TT}(\pt)_{\mathrm{loc}}(y^{\frac{1}{2}})(\!(q,Q)\!), \end{align*}
where $Q$ is multi-index variable with respect to some basis of $H_2(X,\Z)$ and the sum runs over all non-zero effective curve classes $\beta$ and integers $n$. Our first conjecture is a closed expression for the above generating series. In the following conjecture, $t_4$ is the equivariant parameter of the $\BC^*$-action on the fibre $\BC$ of $X = Y \times \BC$, $[x] := x^{\frac{1}{2}} - x^{-\frac{1}{2}}$ and $\mathrm{Exp}(\cdot)$ denotes the plethystic exponential \eqref{equ on ple}.
\begin{conj}[Conjecture \ref{conj: PT of G non-abelian}]\label{conj: stable pairs intro}
Let $G< \mathrm{SU}(4)$ be a finite subgroup with elements of age at most 1. Let $X\to \BC^4/G$ be the crepant resolution given by the Nakamura $G$-Hilbert scheme. 
Then there exist orientations such that
\[
\mathcal{Z}^{\PT}_{X, \oO_X}(y,Q,q)= \mathrm{Exp}\Bigg( \sum_{\beta \in H_2(X,\mathbb{Z})}\frac{-P_{1, \beta}(X,\oO_X,t_4) [y]\, Q^{\beta}}{[t_4][y^{\frac{1}{2}} q] [y^{\frac{1}{2}} q^{-1}]} \Bigg).
\]
\end{conj}
The above conjecture states that the partition function is controlled simply by the invariants $P_{1,\beta}(X, \oO_X, t_4)$, which could be thought as the \emph{$K$-theoretic Gopakumar-Vafa invariants} of $X$. 

If $G<\mathrm{SU}(2)$, then $X=S\times \BC^2$, with $S\to \BC^2/G$ the minimal resolution of an ADE singularity, and we say that a curve class $\beta$ corresponds to a \emph{positive root} if it corresponds to a positive root of the root system associated to Dynkin diagram of $G$ (cf.~Definition~\ref{def: positive roots}). 

If $G<\mathrm{SO}(3)$, then $X=Y\times \BC$, with $Y\to \BC^3/G$ the crepant resolution given by the Nakamura $G$-Hilbert scheme. 
Denote the double cover of $G$ by $\widehat{G}<\mathrm{SU}(2)$. By the work of Boissi\`ere-Sarti \cite{BS}, $Y$ admits a  fibration $Y\to \BC$ with central fiber $S_0$ a partial resolution of the ADE singularity $\BC^2 / \widehat{G}$ (cf.~\S \ref{sec: BS}). In other words, there is a map $S\to S_0\hookrightarrow Y$ which contracts some of the irreducible components of the exceptional locus of the minimal resolution $S\to \BC^2/\widehat{G}$. In this case, we say that a curve class $\beta$ in $Y$ corresponds to a \emph{positive root} if it is the image of a curve class in $S$ corresponding to a positive root.

To support Conjecture \ref{conj: stable pairs intro}, we first calculate $P_{n,\beta}(X, \oO_X, y)$ for $n=0,1$ and show the conjecture holds in these cases (Proposition \ref{cor:n=0,1}). As the invariants are defined by localization, we determine the torus fixed locus for $n=0,1$, and prove that it is non-empty only if $\beta$ corresponds to a positive root, in which case the unique fixed point corresponds to the pair $(\oO_X\to \oO_C)$, where $C$ is the Cohen-Macaulay curve in class $\beta$. The stable pair invariants are then computed by analyzing the weight decomposition of the deformation and obstruction spaces at the fixed points. For subgroups $G<\mathrm{SU}(2)$, an explicit weight decomposition is computed (cf.~Lemma \ref{lem:compareExts}). For subgroups  $G<\mathrm{SO}(3)$, the analysis is reduced to the one already carried out by Bryan-Gholampour \cite{BG2} (cf.~Lemma \ref{lemma: Ext of fourfolds and 3-folds}). 

Finally, in the special case when $G$ is abelian, which means that $G\cong \BZ_r$ or $\BZ_2\times \BZ_2$, we verify the conjecture in several cases by implementing the vertex/edge formalism, developed in \cite{CKM}, into a Maple/Mathematica program (see Propositions \ref{prop: vertex formalism Z_r}, \ref{prop: vertex formalism Z_2 Z_2} for details).

\subsection{Donaldson-Thomas invariants of Calabi-Yau 4-orbifolds}

In \S \ref{sec: CY 4 orbifolds}, we study the (degree 0) Donaldson-Thomas partition function of a Calabi-Yau 4-orbifold $\xX$. In this introduction, we discuss the case when $\xX$ is a global quotient stack $[\BC^4 / G]$, where $G$ is a finite \emph{abelian} subgroup of $\mathrm{SU}(4)$ (with no age restriction). Up to unitary conjugation, $G$ can be taken to be a subgroup of $(\mathbb{C}^*)^4 \cap \mathrm{SU}(4)$, where $(\mathbb{C}^*)^4 < \mathrm{SU}(4)$ is the diagonal subgroup. Then the action of $(\mathbb{C}^*)^4$ descends to $[\BC^4 / G]$, which is \emph{toric}.
%Fact. A complex $n \times n$ matrix is unitarily diagonalizable iff it is normal. Example: unitary matrices, hermitian matrices.
%Fact. Any collection of $n \times n$ normal matrices is simultaneously unitarily diagonalizable.

For a $G$-representation $R$,  consider the  moduli space given at the level of $\BC$-points by
 \begin{align*}
   \Hilb^R([\BC^4/G])= \Big\{W\in\Hilb^{\dim R}(\mathbb{C}^4)\,|\,W \subset \mathbb{C}^4\textrm{ }\textrm{is 0-dim.}\textrm{ }G\textrm{-}\textrm{invariant}
\textrm{ }\textrm{and}\textrm{ }H^0(W,\oO_W)\cong R\Big\},  \nonumber
\end{align*}
where $H^0(W,\oO_W)\cong R$ is an isomorphism of $G$-representations. 
This is a special instance of the moduli space of substacks introduced by Olsson-Starr \cite{OS}.  As before, the moduli space $I:=\Hilb^R([\BC^4/G])$ (which is in fact a scheme) is naturally endowed with a twisted virtual fundamental sheaf 
 \[
 \widehat{\oO}_I^{\vir} \in\, K_0\left(I, \mathbb{Z}\left[\tfrac{1}{2}\right]\right), 
 \]
 which depends on a choice of orientation.
 Let  $L$ be a line bundle on $[\BC^4/G]$, in other words a $G$-equivariant line bundle on $\BC^4$. We define the \emph{tautological vector bundle}
\begin{equation*} 
L^{[R]} := \pi_{I *} (\pi_{[\BC^4/G]}^* L \otimes \oO_{\mathcal{Z}})
\end{equation*}
on  $I$,
where $ \mathcal{Z}\subset [\BC^4/G]\times I$ is the universal closed substack and $\pi_{[\BC^4/G]}, \pi_I$ are the natural projections. The fibre over a $\BC$-point  
$\wW\in I$ corresponding  to the $G$-invariant subscheme $W\subset \BC^4$ is $L^{[R]}|_{\wW}=H^0(W, \oO_W)^G$.
We introduce the following \emph{orbifold} $K$-theoretic Donaldson-Thomas (DT) invariants
\[I_{R}([\BC^4/G],L, y):=\chi\left(I, \widehat{\oO}_I^{\vir}\otimes \widehat{\Lambda}^{\mdot} (L^{[R]} \otimes y^{-1})\right)\in K_0^{\TT}(\pt)_{\mathrm{loc}}(y^{ \frac{1}{2}}),\]
 which are again defined by equivariant localization on the proper fixed locus. 
 
 We denote the corresponding Donaldson-Thomas (DT) \emph{partition function}  by
\begin{align*}
\mathcal{Z}^{\DT}_{[\BC^4/G],L}(y,q)&:=\sum_{R}I_{R}([\BC^4/G],L, y)q^R\in K_0^{\TT}(\pt)_{\mathrm{loc}}(y^{\frac{1}{2}})(\!(q)\!),
 \end{align*}
where  $q$ is a multi-index variable and the sum runs over all possible (isomorphism classes of) $G$-representations $R$.

For any finite abelian  subgroup $G<\mathrm{SU}(4)$, regardless of the age of its elements, the $(\mathbb{C}^*)^4$-fixed locus of $\Hilb^R([\BC^4/G])$ is 0-dimensional reduced (Lemma \ref{lemma: fixed locus stack}) and can be indexed by \emph{$G$-coloured solid partitions}.
We develop in \S \ref{sect on orb vertex} an \emph{orbifold vertex formalism} that computes the above $K$-theoretic invariants, which can be seen purely combinatorially as a refined counting-measure of coloured solid partitions.

We propose explicit conjectural formulae for the above $K$-theoretic partition functions for the groups $\BZ_r<\mathrm{SU}(2)$ and $\BZ_2\times \BZ_2< \mathrm{SO}(3)$ 
(Conjectures \ref{conj: closed formula orbifold vertex}, \ref{conj: closed formula orbifold vertex D4}), which generalizes Nekrasov's celebrated formula for $\BC^4$ \cite{Nek}.
By implementing the orbifold vertex formalism into Mathematica, we provide a proof of these conjectures in some examples (cf.~Proposition \ref{prop: r=1}).

 \subsection{Crepant resolution correspondence}
 
 Inspired by string theory \cite{DHVW1, DHVW2},  Bryan-Graber \cite{BGr} proposed a \emph{crepant resolution correspondence} in Gromov-Witten theory (see also the work of Ruan \cite{Ruan}), which relates the Gromov-Witten invariants of an orbifold to the Gromov-Witten invariants of a crepant resolution of its singularities (if it exists). 
  By the conjecture of Maulik-Nekrasov-Okounkov-Pandharipande \cite{MNOP1}, the Gromov-Witten crepant resolution correspondence for Calabi-Yau 3-folds was transferred to the Donaldson-Thomas side by Young and Bryan-Cadman-Young \cite{Young, BCY}. The resulting Donaldson-Thomas crepant resolution correspondence was later proved by Calabrese, Toda, and Beentjes-Calabrese-Rennemo \cite{Cal, Toda, BCR}. 
  More recently, Liu \cite{Liu} addressed the 3-fold crepant resolution conjecture for $K$-theoretic invariants using the language of quasimaps.

Motivated by its 3-fold analogue, 
we formulate a first instance of a $K$-theoretic crepant resolution conjecture for Donaldson-Thomas theory of \emph{Calabi-Yau 4-folds}. For a finite abelian group $G$, we denote its irreducible representations by $R_0, \ldots, R_{|G|-1}$.
\begin{conj}[Conjecture \ref{conj: crepant resolution for CY4}]\label{conj: CRC intro}
Let $G=\BZ_r<\mathrm{SU}(2) < \mathrm{SU(4)}$ or $\BZ_2\times \BZ_2<\mathrm{SO}(3) < \mathrm{SU(4)}$.
Denote by $X \to \BC^4/G$ the crepant resolution given by the Nakamura $G$-Hilbert scheme and let $\xX = [\BC^4 / G]$. Then there exist orientations such that
\begin{align*}
    \mathcal{Z}^{\DT}_{\xX,\oO_{\xX}}(y,q_0,\ldots,q_{|G|-1})=\mathcal{Z}^{\DT}_{X,\oO_X}(y,0,q)\cdot \mathcal{Z}^{\PT}_{X,\oO_X}(y,Q,q)\cdot \mathcal{Z}^{\PT}_{X,\oO_X}(y,Q^{-1},q),
\end{align*}
under the change of variables $Q^{\beta_i}=q_i$ for $i=1,\dots, |G|-1$, $q=q_0\cdots q_{|G|-1}$, where $\beta_i$ is the curve class corresponding to the irreducible $G$-representation $R_i$ by Reid's generalized McKay correspondence \eqref{equ on mckay}.
\end{conj}
We note that Conjecture \ref{conj: CRC intro} cannot be checked order by order, due to the inversion of the curve-counting parameter $Q\mapsto Q^{-1}$ on the right-hand side. However, we show that the conjectural expression for stable pair invariants of the crepant resolution of $\BC^4/G$ (Conjecture \ref{conj: stable pairs intro}) matches with the conjectural expression of orbifold Donaldson-Thomas invariants of the quotient stack $[\BC^4/G]$ (Conjectures \ref{conj: closed formula orbifold vertex}, \ref{conj: closed formula orbifold vertex D4}) via the above crepant resolution correspondence (Theorem \ref{thm on crc}).

To make a precise formulation of the crepant resolution conjecture  in our local setting, we need to restrict to  (1) abelian groups, otherwise the left-hand side cannot be defined by equivariant localization, and (2) to groups with elements of age at most one to prevent the crepant resolution to contract proper subvarieties of dimension larger than one. However, we explain in Remark \ref{rmk: extending CRC global} the necessary ingredients required to generalize Conjecture \ref{conj: CRC intro} to more general (possibly projective) Calabi-Yau 4-orbifolds.
We hope to address the more general case in a future exploration.

\subsection{Relation to string theory}

Donaldson-Thomas invariants of $\BC^4$ and their $K$-theoretic refinement appear in string theory as the result of supersymmetric localization of $U(1)$ super-Yang-Mills theory with matter on $\BC^4$. This theory was introduced by Nekrasov-Piazzalunga \cite{Nek, NP} using an ADHM-type construction describing the quantum mechanics of a system of D0-D8 branes. It was subsequently extended to orbifolds $[\BC^4/G]$ by Bonelli-Fasola-Tanzini-Zenkevich \cite{BFTZ} (see also \cite{Ki}). 
In three dimensions, Cirafici \cite{Cir} recently interpreted the generating function of $K$-theoretic Donaldson-Thomas invariants of Calabi-Yau 3-orbifolds  as the M2-brane index in M-theory (following the ideas of Nekrasov-Okounkov \cite{NO}).

\subsubsection{Szabo-Tirelli} 

Our present work is complementary to, but independent of, recent papers by Szabo-Tirelli  \cite{ST1, ST2}, which address the  study of \emph{higher rank} cohomological Donaldson-Thomas theory on a toric Calabi-Yau 4-orbifold arising as  the quotient of $\mathbb{C}^4$ by a finite abelian subgroup $G$ of $\mathrm{SU}(4)$. Their study is from the perspective of instanton counting in cohomological gauge theory on a noncommutative crepant resolution of the quotient singularity in quantum field theory. 
They develop a vertex formalism which computes the partition function of cohomological orbifold Donaldson-Thomas invariants and reproduces \eqref{eqn: inv for cohom} in the rank one case. In particular, they propose explicit formulae for the orbifold Donaldson-Thomas invariants of $[\BC^4/\BZ_r]$ and $ [\BC^4/\BZ_2\times \BZ_2]$ \cite[Prop.~4.7, Conj.~5.10]{ST2} which agree with our expressions in Corollary \ref{cor:dimred2}.

\subsection*{Acknowledgements}

We are grateful to Richard Szabo and Michelangelo Tirelli for useful discussions and for sharing an early version of their work \cite{ST2}, and to Noah Arbesfeld, Francesca Carocci, Amin Gholampour for helpful conversations.
Y.~C. thanks Miles Reid for an enlightening discussion on crepant resolutions on 4-folds many years ago. S.~M.~is particularly indebteded to Michele Cirafici, Nadir Fasola, Michele Graffeo, and Andrea Sangiovanni for many conversations on crepant resolutions.
Y.~C.~is partially supported by RIKEN Interdisciplinary Theoretical and Mathematical Sciences
Program (iTHEMS), JSPS KAKENHI Grant Number JP19K23397 and Royal Society Newton International Fellowships Alumni 2021 and 2022. 
M.K.~was supported by NWO grant VI.Vidi.192.012.
S.M. was partially supported by  NWO grant TOP2.17.004 and the Chair of Arithmetic Geometry, EPFL.

\section{Basics on finite subgroups of $\mathrm{SU}(4)$}

\subsection{Generalized McKay correspondence}

Consider the special unitary group $\mathrm{SU}(4)$ with its standard action, by matrix multiplication, on $\C^4$. For a finite subgroup 
$$G<\mathrm{SU}(4), $$
any element $g\in G$ has a finite order, say $r\geqslant  1$. Choose a
basis such that its action on $\mathbb{C}^4$ is diagonalized as 
$$
g=\mathrm{diag}(\epsilon^{a_1},\epsilon^{a_2},\epsilon^{a_3},\epsilon^{a_4}), \,\,\mathrm{with}\,\, 0\leqslant  a_i< r, 
$$
where $\epsilon=\exp(2\pi\sqrt{-1}/r)$ is a primitive $r$th-root of unity  and  $a_i$ are integers. Following Ito-Reid \cite{IR}, one defines the \textit{age} of $g$ by 
$$\mathrm{age}(g):=\frac{1}{r}\sum_{i=1}^4 a_i\in \{0,1,2,3\}. $$
Note that $\mathrm{age}(g)\leqslant 4-4/r$. Therefore an element of order $r\leqslant 3$ has age $\leqslant 2$. 

Reid's generalized McKay correspondence \cite{Reid} predicts a relation between the geometry of a crepant resolution $X\to \mathbb{C}^4/G$ of the quotient singularity (if it exists) and the representation theory of $G$
\begin{align}\label{equ on mckay}
\mathrm{basis}\,\, \mathrm{of}\,\,H_{2k}(X,\mathbb{Z}) \stackrel{1-1}{\longleftrightarrow}   \mathrm{Conj}_k(G),
\end{align}
where $\mathrm{Conj}_k(G)$ denotes the set of conjugacy classes of age $k$ elements of $G$. 
The above generalized McKay correspondence 
was proposed in any dimension in \cite{Reid2}, and proved 
by Batyrev and Denef-Loeser \cite{B,DL} using non-Archimedean methods and motivic integration.
We remark that the generalized McKay correspondence only makes sense if a crepant resolution exists.  
Unlike the 2- and 3-dimensional cases, where crepant resolutions always exist \cite{Ito1, Ito2, Ito3, Mar, MOP, Roan, BKR}, crepant resolutions may not exist in general for $\mathbb{C}^4/G$ \cite{MMM, Reid}.
They have been mostly studied when $G$ is abelian \cite{DHZ1, DHZ2, Mu, Sa, Sato} (we recall some examples in \S \ref{sec: examples}).

\subsection{Subgroups with ages at most one} \label{sec:ageatmostone}

We refer to \cite{HH} for a classification of all finite subgroups of $\mathrm{SU}(4)$. 
Here we classify all finite subgroups $G< \mathrm{SU}(4)$ with elements of age less than or equal to one, generalizing \cite[Lem.~24]{BG1}. In order to state our result, we fix the embedding $\mathrm{SU}(2) < \mathrm{SU}(4)$ by acting trivially on the third and fourth coordinate of $\C^4$. Similarly, we fix the embedding $\mathrm{SU}(3) < \mathrm{SU}(4)$ by acting trivially on the fourth coordinate of $\C^4$. We also consider the standard embedding of the special orthogonal group $\mathrm{SO}(3) < \mathrm{SU}(3)$.

\begin{prop}\label{classify g=1 abelian}
Let $G< \mathrm{SU}(4)$ be a finite subgroup such that all its elements have age $\leqslant 1$. Then, up to unitary conjugation, $G< \mathrm{SO}(3)<\mathrm{SU}(4)$ or  $G< \mathrm{SU}(2)<\mathrm{SU}(4)$.
\end{prop}
\begin{proof} 
For an order $r$ non-trivial element $g\in G$, write
\begin{align}\label{eqn: eigenvalues g}
    g=\mathrm{diag}(\epsilon^{a_1},\epsilon^{a_2},\epsilon^{a_3},\epsilon^{a_4}), 
\end{align}
where $\epsilon $ is a primitive $r$-root of unity,  $0\leqslant  a_i< r$ and $\sum_ia_i \equiv 0\, \mathrm{(mod} \, r) $. Then precisely two $a_i$'s are zero. In fact, denote by $\delta$ the number of indices $i$ such that  $a_i\neq 0$. We have
\[\mathrm{age}(g^{-1})=\frac{1}{r}\sum_{i:\, a_i\neq 0}(r-a_i)=\delta-\mathrm{age}(g),\]
by which we conclude that $\delta=2$.

Let $V$ denote the $4$-dimensional representation induced by $G< \mathrm{SU}(4)$. By \eqref{eqn: eigenvalues g}, the action of $g$ on $V$ has eigenvalues $\{1,1,\epsilon^{k}, \epsilon^{-k}\}$, for some  integer $k$. We consider the character formula
\begin{align*}
    \chi_V(g^2)&=1+1+\epsilon^{2k}+\epsilon^{-2k}\\
    &= 1+(1+\epsilon^k+\epsilon^{-k})(-1+\epsilon^k+\epsilon^{-k})\\
    &= 1+(\overline{\chi}_V(g)-1)(\chi_V(g)-3),
\end{align*}
where we used that $\overline{\chi}_V(g)=\chi_{V}(g)$. Therefore
\begin{align}\label{eqn: FH}
    \frac{1}{|G|}\sum_{g\in G}\chi_{V}(g^2)=\frac{1}{|G|}\sum_{g\in G}\left(1+(\overline{\chi}_V(g)-1)(\chi_V(g)-3)\right).
\end{align}
Following \cite[\S 3.5]{FH}, denote by $r,c,q,t$ respectively the number of real, complex, quaternionic and trivial irreducible subrepresentations of $V$. By \cite[\S 2.1,~Prop.~2.1 \& \S 3.5,~Ex.~3.38]{FH}, we have 
$$\frac{1}{|G|}\sum_{g\in G}\chi_{V}(g^2)=r-q. $$
The orthogonality of characters \cite[\S 2.2,~Thm.~2.12]{FH} implies that 
\begin{align*}
    \frac{1}{|G|}\sum_{g\in G}\overline{\chi}_V(g)\chi_V(g) \geq r+c+q,\quad \frac{1}{|G|}\sum_{g\in G}\overline{\chi}_V(g)=t.
\end{align*}
Using \eqref{eqn: FH}, we obtain the relation
\[c+2q\leq 4t-4.\]
Clearly $t=0$ leads to a contradiction, therefore either $t=1$ and $c=q=0$, which implies that $V$ is a real representation,  in other words $G< \mathrm{SO}(3)< \mathrm{SU}(4)$, or $t\geq 2$, which implies that $G< \mathrm{SU}(2)< \mathrm{SU}(4)$.
\end{proof}
Recall that for any finite subgroup $G < \mathrm{SU}(d)$ acting by matrix multiplication on $\BC^d$, we denote by $G$-$\Hilb(\mathbb{C}^d)$ the Nakamura Hilbert scheme parametrizing $G$-invariant closed subschemes $Z \subset \mathbb{C}^d$ satisfying $H^0(Z,\oO_Z) \cong \BC[G]$, as $G$-representations, where $\BC[G]$ denotes the regular representation of $G$ \cite{Nak}. It comes with a map $G$-$\Hilb(\mathbb{C}^d) \to \mathbb{C}^d/G$ which is a \emph{candidate} for a crepant resolution. 

Let $G< \mathrm{SU}(3)< \mathrm{SU}(4)$ be a finite subgroup acting trivially on the 4-th coordinate of $\mathbb{C}^4$ (such as in this proposition). Then we have  
$$ \mathbb{C}^4/G\cong \mathbb{C}^3/G\times \mathbb{C}. $$
Moreover $X=G$-$\Hilb(\mathbb{C}^4) \cong G$-$\Hilb(\mathbb{C}^3) \times \mathbb{C}$. It is known that $G$-$\Hilb(\mathbb{C}^3)$ is irreducible and provides a crepant resolution of $\mathbb{C}^3/G$ \cite{BKR}. Hence $X$ is irreducible and provides a crepant resolution of $\mathbb{C}^4/G$.

\subsection{Other examples}\label{sec: examples}

In general, non-identity elements in $G< \mathrm{SU}(3)< \mathrm{SU}(4)$ are not necessarily of age one.
\begin{example}\label{exam: Z_3}
Let $\epsilon=\exp(2\pi\sqrt{-1}/3)$. Then 
\begin{align*}
G=\Big\{\mathrm{diag}(\epsilon^i,\epsilon^i,\epsilon^i,1)\,\Big|\,i=0,1,2  \Big\}\cong \BZ_3 \end{align*} 
has elements of age $0,1,2$. Moreover
\emph{G}-$\Hilb(\mathbb{C}^3) \times \C = K_{\mathbb{P}^2}\times \C$ is a crepant resolution of $\C^4/G \cong \mathbb{C}^3 / G \times \mathbb{C}$.
\end{example}
There are examples of abelian subgroups $G< \mathrm{SU}(4)$ such that all elements have age $\leqslant  2$, 
$G\nsubseteq \mathrm{SU}(3)$, and 
$\C^4/G$ has a crepant resolution.
\begin{example}[{\cite[Exam.~5.4]{Reid}}]
Consider the following order $8$ subgroup of $\mathrm{SU}(4)$
\begin{align*}G=\big\{\pm\mathrm{id},\,\, \pm\mathrm{diag}(-1,-1,1,1),\,\,\pm\mathrm{diag}(1,-1,-1,1),\,\, \pm\mathrm{diag}(-1,1,-1,1)\big\}\cong \mathbb{Z}_2^{\times 3}.  \end{align*} 
Here $\mathrm{age}(\mathrm{diag}(-1,-1,-1,-1))=2$ and all other non-identity elements have age one.
Then $\C^4/G$ has crepant resolutions, none of which are given by the $G$-Hilbert scheme.
\end{example}
The above example can be generalized to the following series of examples.
\begin{example}[{\cite[Thm.~4.1~(d)]{Sa}, \cite[Thm.~1.3~(i)]{HIS}, \cite{Mu}}] ${}$ \\
For any $r\in \mathbb{Z}_{\geqslant 2}$, let $\epsilon=\exp(2\pi\sqrt{-1}/r)$ and define $G< \mathrm{SU}(4)$ as the subgroup generated by 
$$\mathrm{diag}(\epsilon,1,1,\epsilon^{r-1}), \,\, \mathrm{diag}(1,\epsilon,1,\epsilon^{r-1}), \,\, \mathrm{diag}(1,1,\epsilon,\epsilon^{r-1}). $$
Then $G\cong \mathbb{Z}_r\times \mathbb{Z}_r \times \mathbb{Z}_r$ and 
$$\C^4/G\cong \big\{(x_1,x_2,x_3,x_4,y)\,|\, x_1x_2x_3x_4=y^r \big\}. $$ 
By \cite[Thm.~1.2]{DHZ1}, it has a crepant resolution. By \cite{Mu}, \emph{G}-$\Hilb(\mathbb{C}^4)$ is a resolution of singularity $\C^4/G$ but not a crepant one.
Note that when $r=4$, $\mathrm{diag}(\epsilon,\epsilon,\epsilon,\epsilon)\in G$ and the age of $\mathrm{diag}(\epsilon^3,\epsilon^3,\epsilon^3,\epsilon^3)$ is $3$.  
\end{example}

\begin{example}[{\cite[Props.~3.1,~3.2~(i)]{Sato}, \cite[Thm.~4.1~(b)]{Sa}}]  ${}$ For any $r\in \mathbb{Z}_{\geqslant 2}$, let $\epsilon=\exp(2\pi\sqrt{-1}/r)$ and define $G< \mathrm{SU}(4)$ as the subgroup generated by 
$$\mathrm{diag}(\epsilon,\epsilon,1,\epsilon^{r-2}), \,\, \mathrm{diag}(1,1,\epsilon,\epsilon^{r-1}). $$
Then $G\cong \mathbb{Z}_r\times \mathbb{Z}_r$ and $\C^4/G$ has a crepant resolution. 
If $r$ is even, 
\emph{G}-$\Hilb(\mathbb{C}^4)$ is a crepant resolution. If $r$ is odd, \emph{G}-$\Hilb\mathbb{C}^4$ is a blow-up of a certain crepant resolution. 
Note that when $r=4$, $\mathrm{diag}(\epsilon,\epsilon,\epsilon,\epsilon)\in G$ and the age of $\mathrm{diag}(\epsilon^3,\epsilon^3,\epsilon^3,\epsilon^3)$ is $3$.  
\end{example}
Similar to Example \ref{exam: Z_3}, we have:
\begin{example}
Let $\epsilon=\exp(2\pi\sqrt{-1}/4)$ and $G=\mathbb{Z}_4$ acting on $\C^4$ by 
$\mathrm{diag}(\epsilon^i,\epsilon^i,\epsilon^i,\epsilon^i)$ with $i=0,1,2,3$, which have age $0,1,2,3$ respectively. Then $K_{\mathbb{P}^3}$ is a crepant resolution of $\C^4/G$. 
\end{example}
Here is an example for which no crepant resolution exists. 
\begin{example}
Let $G\cong \mathbb{Z}_2$ act on $\C^4$ by $\mathrm{diag}(\pm1,\pm1,\pm1,\pm1)$. Then $\C^4/G$ does not have a crepant resolution. 
By Eqn.~\eqref{equ on mckay}, any crepant resolution will have vanishing second cohomology which is not possible. 
\end{example}

\section{$\mathrm{DT_4}$ virtual structures} \label{sec: DT4 virtual struc}

\subsection{Virtual classes and virtual structure sheaves}\label{sec: 3.1}

Let $X$ be a Calabi-Yau 4-fold, i.e., a smooth quasi-projective variety with trivial canonical bundle $K_X\cong \oO_X$.
Denote by $M$ one of the following two moduli spaces:
\begin{itemize}
\item $I:=\Hilb^n(X,\beta)$ is the Hilbert scheme of proper closed subschemes $Z \subseteq X$ (or, equivalently, their ideal sheaves $I_Z \subset \mathcal{O}_X$) of dimension $\leqslant  1$ satisfying $[Z] = \beta$ and $\chi(\oO_Z)=n$,
\item $P:=P_n(X,\beta)$ is  the moduli space of stable pairs $(F,s)$, where $F$ is a pure 1-dimensional sheaf on $X$ with proper scheme theoretic support in class $\beta$, $\chi(F) = n$, and $s \in H^0(X,F)$ has 0-dimensional cokernel. 
\end{itemize} 
It is well-known that $M$ is endowed with an obstruction theory
\begin{equation}\label{eqn: obstruction theory schemes}
     \mathbb{E}_M=\mathbf{R}\pi_{M*}\mathbf{R}\HOM(\eE, \eE)_0[3]\to \mathbb{L}_{M}, 
\end{equation}
where $(\cdot)_0$ denotes the trace-free part,  
   $\pi_M:X\times M\to M$ is projection, and $\eE$ is either the universal ideal sheaf  $\mathcal{I}\subset \oO_{X\times I}$ or the universal stable pair $\BI^\bullet=(\oO_{X\times P}\to \BF)\in D^{b}(X\times P)$. Furthermore, $\mathbb{L}_{M} \in D^{\mathrm{b}}(M)$ denotes the truncated cotangent complex of $M$.
Moreover, the obstruction theory \eqref{eqn: obstruction theory schemes} is symmetric, i.e., there is an isomorphism $\theta : \mathbb{E}_M^{\vee}[2] \cong \mathbb{E}_M$ induced by Grothendieck-Verdier duality for $\pi : X \times M \to M$, such that $\theta^\vee[2] = \theta$.\footnote{For non-compact $X$, this requires taking a smooth compactification $j : X \hookrightarrow \overline{X}$, i.e., an open immersion into a smooth projective 4-fold. For example, in the stable pairs case we obtain an open subscheme $P := P_n(X,\beta) \subset \overline{P} := P_n(\overline{X},j_*\beta)$. The universal object $\overline{\mathbb{I}^{\mdot}} = (\oO_{\overline{X} \times \overline{P}} \to \overline{\mathbb{F}})$ is related to the universal object $\mathbb{I}^\mdot = (\oO_{X \times P} \to \mathbb{F})$ by adjunction, i.e., $\overline{\mathbb{I}^{\mdot}}|_{ \overline{X} \times P} \cong (\oO_{X \times \overline{P}} \to (j \times \mathrm{id})_! \mathbb{F} )$. Then $\mathbb{E}_P$ is the restriction of a 3-term obstruction theory on $\overline{P}$ and the symmetric form $\theta : \mathbb{E}_P^\vee[2] \cong \mathbb{E}_P$ is induced by Grothendieck-Verdier duality for $\overline{\pi} : \overline{X} \times \overline{P} \to \overline{P}$ and restriction to $P$.}  
By \cite{OT}, there is virtual class and a (twisted) virtual structure sheaf 
\begin{align*}
    [M]^{\vir} \in\, A_{n}\left(M,\mathbb{Z}\left[\tfrac{1}{2}\right]\right), \quad \widehat{\oO}_{M}^{\vir} \in\, K_0\left(M, \mathbb{Z}\left[\tfrac{1}{2}\right]\right),
\end{align*}
which depend on the choice of orientation \cite{CGJ, Bojko, CL2}. The virtual class coincides via the cycle map with the (real) virtual classes constructed by \cite{BJ} (and \cite{CL1} in special cases) after inverting 2.
%Using Hironaka, one can show the following. For any irreducible smooth quasi-projective variety $X$, there exists an open immersion into an irreducible smooth projective variety $\overline{X}$.  

\subsection{Virtual localization formula}\label{sec: virtual localization}

Let $\mathbf{T}$ be an algebraic torus acting regularly on a Calabi-Yau 4-fold $X$ such that $\mathbf{T}$ preserves the Calabi-Yau volume form --- the main source of examples are toric varieties and non-compact local varieties. By \cite{Ric}, the $\TT$-action lifts to $M$ and to a $\mathbf{T}$-equivariant structure on the universal sheaf $\eE$. Then $\mathbb{E}$ is $\mathbf{T}$-equivariant and $\TT$-equivariantly self-dual by relative Serre duality. Suppose there exists a smooth projective 4-fold $\overline{X}$ with $\TT$-action, and a $\TT$-equivariant open immersion $X \hookrightarrow \overline{X}$.\footnote{This assumption is satisfied in all the examples considered in this paper.}

Denote by 
$T_M^{\vir} := \mathbb{E}_M^\vee$ the \emph{virtual tangent bundle}. By \cite[Eqn.~(111)]{OT}, there exists an induced self-dual obstruction theory on the $\mathbf{T}$-fixed locus $M^\mathbf{T}$, with virtual tangent bundle $T^{\vir}_{M^\mathbf{T}}= (T_M^{\vir}|_{M^\mathbf{T}})^{\fix}$, i.e., the $\mathbf{T}$-fixed part of the restriction of the virtual tangent bundle. Denote the movable part by $N^{\vir}:=(T_M^{\vir}|_{M^\mathbf{T}})^{\mov}$, which is called the \emph{virtual normal bundle}. We denote their classes in the $\mathbf{T}$-equivariant Grothendieck group of locally free sheaves on $M^{\TT}$ by the same symbols
\[T^{\vir}_{M^\mathbf{T}}, \quad N^{\vir} \in K^0_\TT(M^{\TT}).\] 
For a fixed orientation for $T_M^{\vir}$, and any choice of orientation for $N^{\vir}$, one obtains an induced orientation for $T^{\vir}_{M^\mathbf{T}}$ (simply because $\det(T^{\vir}_{M^\mathbf{T}}) \cong \det(T_M^{\vir}|_{M^\TT}) \otimes \det(N^{\vir})^{-1})$.
Orientations on $N^{\vir}$ always exist, e.g., take any 1-parameter subgroup of $\TT$ and split $N^{\vir}$ into positive and their dual negative weight spaces (see \cite{OT} for details).
Recall the following  virtual localization formula.
\begin{thm}\emph{(\cite[Thm.~7.3]{OT})}\label{thm: OT virtual localization}
Denote by $\iota:M^\mathbf{T}\hookrightarrow M$ the inclusion. Then
\[
\widehat{\oO}_{M}^{\vir}=\iota_*\frac{\widehat{\oO}_{M^\TT}^{\vir}}{\sqrt{\mathfrak{e}^\mathbf{T}}(N^{\vir})}.
\]
In particular, if $M$ is a  proper scheme, then for any $K$-theory class $V\in K_0^\mathbf{T}(M)$ 
\begin{align}\label{eqn: virtual localization chi}
\chi(M, V\otimes \widehat{\oO}_{M}^{\vir})=\chi\left(M^\TT, \frac{V|_{M^\TT}\otimes \widehat{\oO}_{M^\TT}^{\vir}}{\sqrt{\mathfrak{e}^{\mathbf{T}}}(N^{\vir})}\right).    
\end{align}
\end{thm}
We pause a moment to explain the notation in Theorem \ref{thm: OT virtual localization}. For a $\mathrm{SO}(2r,\mathbb{C})$-bundle $E$ on $M$, one defines the \emph{$K$-theoretic Euler class}
\begin{align*}
    \mathfrak{e}(E):=\Lambda^\bullet E^*:=\sum_{i=0}^{2r}(-1)^i \Lambda^i E^*\in K^0(M)
\end{align*}
and its \emph{square root}
\begin{align*}
    \sqrt{\mathfrak{e}}(E)\in K^0\left(M, \mathbb{Z}\left[\tfrac{1}{2}\right]\right),
\end{align*}
which satisfies
\begin{align*}
  \mathfrak{e}(E) =(-1)^{r} (\sqrt{\mathfrak{e}}(E))^2.
\end{align*}
If $T\subset E$ is a maximal isotropic subbundle (with respect to the quadratic pairing on $E$), one simply has
\begin{align*}
     \sqrt{\mathfrak{e}}(E)=\pm\frac{\Lambda^\bullet T^*}{(\det T^*)^{\frac{1}{2}}}.
\end{align*}
Here the sign is uniquely determined by $T$ and the orientation on $E$. Moreover, the square root of a line bundle is uniquely determined in $K^0\left(M, \mathbb{Z}\left[\tfrac{1}{2}\right]\right)$ as defined in \cite[\S 5.1]{OT}. The definition of $\sqrt{\mathfrak{e}}$  extends to the virtual normal bundle.

If the moduli space $M$ is not proper but $M^\TT$ is proper, then we take \eqref{eqn: virtual localization chi} as the definition of the left-hand side, i.e., we \emph{define} invariants by $\TT$-localization. Here any choice of orientation on $M$ and $N^{\vir}$ induces an orientation on $M^{\TT}$ which we use on the right-hand side of \eqref{eqn: virtual localization chi}.

Computing  the square root Euler class is a challenging task, considering that maximal isotropic subbundles do not always exist. However, the situation simplifies if the $\TT$-fixed locus is reduced and 0-dimensional as studied by \cite{Nek, NP, CK1, CK2, CKM, Mon}.
\begin{defi}
Let $E\in K^0_\TT(\pt)$ be a virtual $\TT$-representation. We say that $T\in K^0_\TT(\pt)$ is a \emph{square root} of $E$ if
\[
E=T+\overline{T}\in K_\TT^0(\pt),
\]
where $\overline{(\cdot)}$ denotes the dual $\TT$-representation.
\end{defi}
For an irreducible $\TT$-representation $t^{\mu}$, we set 
$$[t^\mu]:=t^{\frac{\mu}{2}}-t^{-\frac{\mu}{2}}$$ 
and extend it  to any virtual $\TT$-representation $V=\sum_\mu t^\mu- \sum_\nu t^\nu$ by
\begin{align*}
    [V]=\frac{\prod_\mu[t^\mu]}{\prod_\nu [t^\nu]},
\end{align*}
where we assume that no weights $\nu$ are trivial. Here we use multi-index notation $t$ for the equivariant parameters of $\TT$, i.e., $t = (t_1, \ldots)$ and $t^\mu = \prod_i t_i^{\mu_i}$.

Assume now that the $\TT$-fixed locus $M^\TT$ is 0-dimensional and reduced. 
Then the virtual tangent bundle at every fixed point $Z\in M^\TT$ is a $\TT$-representation with no positive $\TT$-fixed term, which admits a (non-unique) ``square root''
%We use here that $\rk T_M^{\vir}$ is even.
\begin{align*}
    T^{\vir}_{M,Z}=T^{\half}_Z+\overline{T^{\half}_Z} \in K^0_{\TT}(\pt).
\end{align*}
One computes its square root Euler class as
\begin{align*}
    \sqrt{\mathfrak{e}}(T_{M,Z}^\vir)&= \pm \frac{\Lambda^\bullet T^{\half,\vee}_Z}{(\det T^{\half,\vee}_Z)^{\frac{1}{2}}}\\
    &=\pm [T^{\half}_Z].
\end{align*}
Here the sign $\pm 1 $ depends on the choice of orientations of $T_M^{\vir}, N^{\vir}$ and the choice of the square root $T^{\half}_Z$ at the $\TT$-fixed point $Z$, 
and in the last equality we use \cite[\S 6.1]{FMR}. In this case, Theorem \ref{thm: OT virtual localization} reduces to
\begin{align}\label{eqn: localization with []}
\chi(M, V\otimes \widehat{\oO}_{M}^{\vir})=\sum_{Z\in M^\TT}\pm  [-T^{\half}_Z]\cdot V|_Z.
\end{align}
For more discussions on the \emph{signs}, see \cite[Rmk.~1.18]{CKM}, \cite{Mon}.

\begin{rmk}
Notice that in \eqref{eqn: localization with []} we used $T^{\vir}_{M,Z}$ rather than the virtual normal bundle $N^{\vir}_{Z}$. In fact, if $T^{\vir}_{M,Z}$ has no negative $\TT$-fixed terms, then $T^{\vir}_{M,Z}=N^{\vir}_{Z} \in K_\TT^0(\pt)$, while if  $T^{\vir}_{M,Z}$ has a negative $\TT$-fixed term we have $[T_{M,Z}^{\vir}]=0$. 
\end{rmk}

\subsection{Nekrasov genus}

Let $X$ be a Calabi-Yau 4-fold and $L$ a line bundle on $X$. We define the \emph{tautological complex}
\begin{equation*} 
L^{[n]} := \mathbf{R}\pi_{I*} (\pi_X^* L \otimes \oO_{\mathcal{Z}}), \quad \mathbf{R}\pi_{P*} (\pi_X^* L \otimes \mathbb{F}),
\end{equation*}
on the moduli spaces $I$ and $P$ (cf. \S \ref{sec: 3.1}),
where $ \mathcal{Z}\subset X\times I$ is the universal closed subscheme, $\mathbb{I}^{\mdot} =(\oO_{P \times X} \rightarrow \mathbb{F})\in D^{\mathrm{b}}(X\times P) $ is the universal stable pair, and $\pi_X, \pi_I, \pi_P$ are the projections.

In \cite{CKM}, we defined the following \emph{Nekrasov genera} extending definitions for Hilbert schemes of points on $\mathbb{C}^4$ used in the physics literature \cite{Nek, NP}.
\begin{defi}\textrm{(\cite[Def.~0.2]{CKM})} \label{Nekgen}
Let $X$ be a projective Calabi-Yau 4-fold, with  a trivial $\BC^*$-action and let $\oO_X \otimes y$ be the trivial line bundle with non-trivial $\BC^*$-equivariant 
structure corresponding to the irreducible character $y$.
For any line bundle $L$ on $X$, we define the \emph{Nekrasov genus} by
\begin{align*}
I_{n,\beta}(X,L, y) := \chi\left(I, \widehat{\oO}^{\vir}_I \otimes \widehat{\Lambda}^{\mdot} (L^{[n]} \otimes y^{-1}) \right)\in \mathbb{Z}\left[\tfrac{1}{2}\right](y^{\frac{1}{2}}),
\end{align*}
where $\widehat{\Lambda}^\bullet(\cdot):=\det(\cdot)^{-1/2}\otimes \Lambda^{\bullet}(\cdot)$ and we define $P_{n,\beta}(X,L, y)$ analogously replacing $I$ by $P$.\footnote{The square root may not exists as a genuine line bundle but it uniquely exists as a class in the $K$-group if we invert 2 \cite[Rmk.~5.2]{OT}.} Suppose $X$ is not projective but endowed with a $\TT$-action preserving the Calabi-Yau volume form and such that the fixed locus $I^{\TT}$ is proper. Suppose $L$ is a $\TT$-equivariant line bundle on $X$. Then we define the invariants by means of the virtual localization formula \eqref{eqn: virtual localization chi}
\begin{align}\label{eqn: inv loc ordinary}
I_{n,\beta}(X,L, y) := \chi\left(I^\TT,\frac{\widehat{\oO}^{\vir}_{I^\TT}}{\sqrt{\mathfrak{e}^{\TT}}(N^{\vir})} \otimes \widehat{\Lambda}^{\mdot} (L^{[n]}|_{I^\TT} \otimes y^{-1}) \right)\in K_0^{\TT}(\pt)_{\mathrm{loc}}(y^{\frac{1}{2}}),
\end{align}
and similarly for $P_{n,\beta}(L, y)$, 
where 
\begin{equation}\label{equ for k0 loc} K_0^{\TT}(\pt)_{\mathrm{loc}}:=\BQ\left(t_1^{\frac{1}{2}},\dots,  t_r^{\frac{1}{2}}\right), \,\,\, r=\rank \TT. \end{equation}
\end{defi}
We define the $K$-theoretic Donaldson-Thomas (DT) and Pandharipande-Thomas (PT) \emph{partition functions} by
\begin{align}\label{pt partition function}
\mathcal{Z}^{\DT}_{X,L}(y,Q,q)&:=1+\sum_{\beta>0,n\in\mathbb{Z}}I_{n,\beta}(X,L, y)\,Q^{\beta}q^n\in K_0^{\TT}(\pt)_{\mathrm{loc}}(y^{\frac{1}{2}})(\!(q,Q)\!),\\ 
\label{dt partition function}
\mathcal{Z}^{\PT}_{X,L}(y,Q,q)&:=1+\sum_{\beta>0,n\in \mathbb{Z}}P_{n,\beta}(X,L, y)\,Q^{\beta}q^n\in K_0^{\TT}(\pt)_{\mathrm{loc}}(y^{\frac{1}{2}})(\!(q,Q)\!), \end{align}
where $Q^\beta$ is multi-index notation with respect to some basis of $H_2(X,\Z)$. The sums run over all non-zero effective curve classes $\beta$ on $X$ and integers $n$.

\section{Stable pairs on crepant resolutions}

\subsection{Nakamura $G$-Hilbert schemes}\label{sect on Nak G Hilb}

Let $G$ be a finite subgroup of $\mathrm{SU}(2)$. There is an associated action of $G$ on $\BC^2$, which extends to an action on $\BC^4$ by trivially acting on the third and fourth coordinates. By \cite[Thm.~1.3]{IN2}, the Nakamura $G$-Hilbert scheme $S=G\textrm{-Hilb}(\mathbb{C}^2)$ \cite{Nak} realizes the (minimal) crepant resolution $S\to \BC^2/G$ 
of the ADE singularity $\BC^2/G$. Therefore we have a crepant resolution 
\[
X=G\textrm{-Hilb}(\mathbb{C}^4) \cong  S\times \BC^2\to \BC^4/G.
\]
Then $X$ is a Calabi-Yau 4-fold, endowed with a natural $(\BC^*)^3$-action induced by the lift of  the diagonal action on $\BC^2$ (see also \cite[Prop.~8.2]{B})
\begin{align}\label{eqn: diagonal action C^2}
  (t,t_3,t_4)\cdot (x_1, x_2, x_3, x_4)=(tx_1, tx_2, t_3x_3, t_4x_4).  
\end{align}
The restriction to the subtorus 
$$\TT_0=\{t^2 t_3t_4=1\}\subset (\BC^*)^3$$ 
preserves the Calabi-Yau volume form of $\BC^4$ and $X$.

Similarly, if $G$ is a \emph{polyhedral group}, namely  a finite subgroup  of $\mathrm{SO}(3)$, there is an associated action of $G$ on $\BC^3$, which  extends to an action on $\BC^4$ by acting trivially on the fourth coordinate. By \cite{BKR}, the Nakamura $G$-Hilbert scheme $Y=G\textrm{-Hilb}(\mathbb{C}^3)$ is irreducible and gives a preferred\footnote{In \cite{CI}, Craw-Ishii show that when $G$ is abelian, other crepant resolutions are given by moduli spaces of $G$-constellations, which generalize Nakamura's moduli spaces of $G$-clusters. See Yamagishi \cite{Yam} for a recent generalization to the non-abelian case.} crepant resolution $Y\to \BC^3/G$ and therefore we obtain a crepant resolution 
\[
X=G\textrm{-Hilb}(\mathbb{C}^4) \cong Y\times \BC\to \BC^4/G.
\]
Then $X$ is a Calabi-Yau 4-fold, endowed with a natural $(\BC^*)^2$-action induced by the lift of the diagonal action on $\BC^3$
\begin{align}\label{eqn: diagonal action}
  (t,t_4)\cdot (x_1, x_2, x_3, x_4)=(tx_1, tx_2, tx_3, t_4x_4).  
\end{align}
The restriction to the subtorus 
$$\TT_1=\{t^{3}t_4=1\}\subset (\BC^*)^2$$ 
preserves the Calabi-Yau volume form of $\BC^4$ and $X$.

If $G$ is furthermore abelian in one of the above cases, i.e.~$G=\BZ_r$ or $\BZ_2\times \BZ_2$, the induced $G$-action commutes with the $(\BC^*)^4$-action  on $\BC^4$, and the crepant resolution  $X\to \BC^4/G$ is a smooth toric variety. The subtorus 
$$\TT_2=\{t_1t_2t_3t_4=1\}\subset (\BC^*)^4$$ 
preserves the Calabi-Yau volume form of $X$ and  refines the diagonal actions \eqref{eqn: diagonal action C^2}, \eqref{eqn: diagonal action}.  Whenever it is clear from the context, we denote  by $\TT$ one of  the tori $\TT_0,\TT_1$ or $\TT_2$ (the maximal one possible in the given context).

\subsection{Boissi\`ere-Sarti}\label{sec: BS}

We briefly summarize the correspondence of \cite{BS} between crepant resolutions of finite subgroups $G<\mathrm{SO}(3)$ and minimal resolutions of ADE singularities.

For every finite subgroup $G<\mathrm{SO}(3)$, let $\widehat{G}<\mathrm{SU}(2)$ be its double cover. Let 
$$Y\to \BC^3/G$$ 
denote the crepant resolution given by the Nakamura $G$-Hilbert scheme and write $S\to \BC^2/\widehat{G}$ for the minimal resolution of singularities. 
The resolution $Y$ admits a fibration 
\begin{equation}\label{equ on fib of Y}\pi: Y\to \BC, \end{equation} 
whose central fibre $S_0=\pi^{-1}(0)$ is a partial resolution of $\BC^2/\widehat{G}$. Namely, there is a map 
\begin{align*}
    f:S\to S_0\subset Y, 
\end{align*}
which contracts some irreducible components of the exceptional locus of $S\to \BC^2/\widehat{G}$. 
Recall that the \textit{reduced McKay quiver} of $\widehat{G}$ is obtained from the McKay quiver of  $\widehat{G}$ by removing the vertex of the trivial representation (ref.~\cite[\S 5]{BS}). There are bijections between nodes of the reduced McKay quiver, simple roots of the root system of $\widehat{G}$, and irreducible components of the exceptional divisor of the minimal resolution of $\BC^2/\widehat{G}$. By definition, the components contracted by $f$ correspond to \emph{binary roots} (see \cite[Fig.~5.1,~5.2]{BS}, where binary roots are the black vertices). 
The remaining (non-binary) roots correspond to the irreducible components of the exceptional locus of $Y \to \mathbb{C}^3 / G$.

Denote by $R^+$ the collection of \emph{positive roots} of the root system associated to $\widehat{G}$. To each positive root, there is an associated curve class in $S$, i.e., we have a map
\begin{align*}
   c: R^+\hookrightarrow H_2(S, \BZ),
\end{align*}
which we compose with the contraction of the binary roots to get
\begin{align*}
    \widetilde{c}:R^+\hookrightarrow H_2(S, \BZ)\to H_2(Y, \BZ).
\end{align*}
\begin{defi}\label{def: positive roots}
A non-zero curve class $\beta\in H_2(S, \BZ)$ (resp.~$H_2(Y, \BZ)$)~\emph{corresponds to a positive root} if $\beta\in c(R^+)$ (resp. $\widetilde{c}(R^+)$).
\end{defi}

\subsection{Stable pair invariants}\label{sec: PT invariants resolution}

Let $G< \mathrm{SU}(4)$ be a finite subgroup with elements of age at most 1. Let $X\to \BC^4/G$ be the crepant resolution given by the Nakamura $G$-Hilbert scheme as in \S \ref{sec:ageatmostone}. 
Furthermore, let $\TT$ be one of the tori $\TT_0,\TT_1$, or $\TT_2$ as described in \S \ref{sect on Nak G Hilb}.
\begin{prop} \label{prop:properfixloc}
Let $\beta\in H_2(X,\BZ)$ be a non-zero curve class and  $n\in \BZ$. Then the $\TT$-fixed locus $P_n(X, \beta)^{\TT}$ is proper. 
\end{prop}
\begin{proof}
We prove the case $G<\mathrm{SO}(3)$. The other cases follow from a similar argument. 
We write $X=Y\times \mathbb{C}$ and we recall the fibration $\pi: Y\to \C$ \eqref{equ on fib of Y}. Hence we have a map 
\begin{equation}\label{fib of x} \pi \times \mathrm{id}_{\mathbb{C}} : X\to \mathbb{C}^2. \end{equation}
For any $\TT=\TT_1$-fixed $[(F,s)] \in P_n(X, \beta)$, 
since $F$ is compactly supported and $\mathbb{C}^2$ is affine, $F$ is set theoretically supported on 
the central fibre $S_0=\pi^{-1}(0)\times \{0\}\subset X$ of the map \eqref{fib of x}. 
Let $\widehat{G}$ be the double cover of $G$ as in \S \ref{sec: BS}, then there is a map 
$$S_0\to \mathbb{C}^2/\widehat{G}, $$
which is a partial resolution of an ADE singularity. As $\mathbb{C}^2/\widehat{G}$ is also affine, we know $F$ is set theoretically supported 
on the exceptional locus, which is proper. As in \cite[Prop.~3.1]{CKM2}, we deduce that $P_n(X, \beta)^{\TT}$ is proper. 
\end{proof}

Recall that for any $\TT$-equivariant line bundle $L$ on $X$, we have a partition function (Definition \ref{Nekgen})
\begin{align*}\mathcal{Z}^{\PT}_{X,L}(y,Q,q):=1+\sum_{\beta,n}P_{n,\beta}(X,L, y)Q^{\beta}q^n\in R(\!(q,Q)\!), \end{align*}
where the ring $R$ is by defined as
\[
R=\begin{cases}
\frac{\BQ(t_1^{\frac{1}{2}}, t_2^{\frac{1}{2}}, t_3^{\frac{1}{2}}, t_4^{\frac{1}{2}}, y^{\frac{1}{2}})}{(t_1t_2t_3t_4-1)} & \mathrm{if}\, G \,\,\mathrm{is} \mbox{ abelian},\\
\frac{\BQ(t^{\frac{1}{2}}, t_3^{\frac{1}{2}}, t_4^{\frac{1}{2}}, y^{\frac{1}{2}})}{( t^2t_3t_4-1)} & \mathrm{if}\, G<\mathrm{SU}(2),\\
\frac{\BQ(t^{\frac{1}{2}}, t_4^{\frac{1}{2}}, y^{\frac{1}{2}})}{(t^3t_4-1)} & \mathrm{if}\, G<\mathrm{SO}(3).
\end{cases}
\]
Inspired by the closed formula for the local resolved conifold \cite[Conj.~0.16]{CKM}, we conjecture a closed formula for the stable pair partition function of the crepant resolution $X\to \BC^4/G$.
\begin{conj}\label{conj: PT of G non-abelian}
Let $G< \mathrm{SU}(4)$ be a finite subgroup with elements of age at most 1. Let $X\to \BC^4/G$ be the crepant resolution given by the Nakamura $G$-Hilbert scheme as in \S \ref{sec:ageatmostone}.
Then there exist orientations such that
\[
\mathcal{Z}^{\PT}_{X, \oO_X}(y,Q,q)= \mathrm{Exp}\Bigg( \sum_{\beta \in H_2(X,\mathbb{Z})}\frac{-P_{1, \beta}(X,\oO_X,t_4) [y]\, Q^{\beta}}{[t_4][y^{\frac{1}{2}} q] [y^{\frac{1}{2}} q^{-1}]} \Bigg).
\]
\end{conj}
Here,  for any formal power series $f(p_1, \ldots, p_r; q_1, \ldots, q_s)$ in $\Q(p_1, \ldots, p_r)[\![q_1, \ldots, q_s]\!]$, such that $f(p_1, \ldots, p_r;0,\ldots,0)=0$, its \emph{plethystic exponential} is defined as 
\begin{align}\label{equ on ple} 
\Exp(f(p_1, \ldots, p_r;q_1, \ldots, q_s)) &:= \exp\Big( \sum_{n=1}^{\infty} \frac{1}{n} f(p_1^n, \ldots, p_r^n;q_1^n, \ldots, q_s^n) \Big)
\end{align}
viewed as an element of $\Q(p_1, \ldots, p_r)[\![q_1, \ldots, q_s]\!]$. 
\begin{rmk}
Motivated by \cite{Nag}, we expect Conjecture \ref{conj: PT of G non-abelian} also holds for $X=Y\times \mathbb{C}$, where $Y$ is a toric Calabi-Yau 3-fold such that all genus $g\geqslant 1$ Gopakumar-Vafa invariants
of $Y$ vanish.
\end{rmk}

\begin{rmk}
Applying the dimensional reduction and cohomological limit explained in \S\ref{sec: Limit of theory}, Conjecture \ref{conj: PT of G non-abelian} reduces to an expression which coincides with the formulae for the stable pair invariants of $Y = S \times \mathbb{C}$, where $S$ is the minimal resolution of an ADE singularity in\cite{GJ}.\footnote{The authors of \cite{GJ} point out a mistake in one of their  proofs, but they expect that their main results nevertheless hold.}
\end{rmk}
In analogy with the PT/GV correspondence \cite{CMT2}, we refer to $P_{1, \beta}(X,\oO_X,y)$ as \emph{K-theoretic Gopakumar-Vafa invariants} of $X$ (compare also 
with \cite{KP, CMT1, CT2} for Gopakumar-Vafa invariants of Calabi-Yau 4-folds defined using primary insertions).
In \S \ref{sec: dim red}, we show that 
$$
P_{1, \beta}(X,\oO_X,t_4) = \left\{\begin{array}{cc} \pm\chi\left(P_{1}(S \times \C,\beta),\widehat{\mathcal{O}}^{\vir}\right) & \mathrm{if} \, G < \mathrm{SU}(2), \\ \pm\chi\left(P_{1}(Y,\beta),\widehat{\mathcal{O}}^{\vir}\right) & \mathrm{if} \, G < \mathrm{SO}(3), \end{array}\right.
$$
where the right-hand side are the $K$-theoretic invariants of the Calabi-Yau \emph{3-folds} $S \times \mathbb{C}$ resp.~$Y$ (defined in Nekrasov-Okounkov \cite{NO}) and $\beta$ is viewed as a curve class on $S \times \mathbb{C}$ resp.~$Y$. Therefore, Conjecture \ref{conj: PT of G non-abelian} fully reduces the $K$-theoretic invariants of the 4-fold to the \emph{$n=1$} $K$-theoretic invariants of the 3-fold.

In the rest of this section, we explain how to calculate $P_{1, \beta}(X,\oO_X,y)$ and we provide evidence for Conjecture \ref{conj: PT of G non-abelian} by verifying it in some cases.

\subsubsection{K-theoretic Gopakumar-Vafa invariants}
We determine $P_{n,\beta}(X,\oO_X, y)$ for $n=0,1$ and $\beta\in H_2(X,\BZ)$. In particular, we show that the  $K$-theoretic Gopakumar-Vafa invariants are non-zero only for curve classes corresponding to positive roots.
To any non-zero curve class $$\beta=\sum_i m_i [C_i]$$ corresponding to a positive root, where $C_i$ are irreducible components of the exceptional locus, we can associate a  unique Cohen-Macaulay curve $C$ having generic multiplicity $m_i$ along $C_i$ (ref.~\cite[\S 3]{BG2}). Specifically, when $I_{C_i}$ is the ideal sheaf of $C_i$ in $S$ (resp.~$S_0$), we define the scheme structure on the corresponding component by $I_{C_i}^{m_i}$.
\begin{lem}\label{lemma: fixed locus stable pairs}
Let $G< \mathrm{SU}(4)$ be a finite subgroup with elements of age at most 1. Let $X\to \BC^4/G$ be the crepant resolution given by the Nakamura $G$-Hilbert scheme as in \S \ref{sec:ageatmostone}.
Then $P_n(X, \beta)^{\TT}$ satisfies the following:
\begin{enumerate}
    \item $P_0(X,\beta)^{\TT}=\varnothing$.
    \item At the level of closed points, $P_1(X, \beta)^{\TT}=\big\{(\oO_X\xrightarrow{s}\oO_C)\big\}$ if $\beta$ corresponds to a positive root, where $C$ is the Cohen-Macaulay curve in class $\beta$ and $s$ is the canonical section of $\oO_C$.
    \item $P_1(X, \beta)^{\TT}=\varnothing$ if $\beta$ does not correspond to a positive root.
\end{enumerate}
\end{lem}
\begin{proof}
We prove the case when $G<\mathrm{SO}(3)$. The other cases follow from a similar argument. 
We write $X=Y\times \mathbb{C}$ with the natural projection $p: X\to Y$. As in \cite[\S 5.2]{CMT1}, for any $\TT=\TT_1$-fixed stable pair $(F,s)$ on $X$ with proper support, 
we have an eigensheaf decomposition 
\begin{equation}\label{equ1}p_*F=\bigoplus_{i}F_i\end{equation}
with respect to the $\TT$-action and a section $s_i:\oO_Y\to F_i$ for each $i$, whose cokernel is 0-dimensional. 

We claim that $\chi(F_i)\geqslant 1$ and equality holds exactly when $s_i$ is surjective and $F_i$ is stable and scheme-theoretically supported on the CM curve corresponding to a positive root. 
In fact, consider the Harder-Narasimhan and Jordan-H\"older filtration 
$$0 = E_0 \subseteq E_1\subseteq E_2\subseteq\cdots\subseteq E_N=F_i, $$ 
where $E_j/E_{j-1}$ are 1-dimensional stable sheaves with 
\begin{equation*}\frac{\chi(E_1)}{\deg (E_1)}\geqslant \frac{\chi(E_2/E_1)}{\deg (E_2/E_1)}\geqslant \cdots \geqslant \frac{\chi(E_N/E_{N-1})}{\deg (E_N/E_{N-1})}.\end{equation*}
Composing $s_i$ with the surjection $E_N\twoheadrightarrow E_N/E_{N-1}$ gives a section $\oO_Y\to E_N/E_{N-1}$ whose cokernel $Q$ is 0-dimensional. 
Let $C$ be the scheme theoretic support of $E_N/E_{N-1}$, then 
\begin{equation*}\chi(E_N/E_{N-1})=\chi(Q)+ \chi(\oO_C)\geqslant \chi(\oO_C). \end{equation*}
Recall the fibration $\pi: Y\to \C$ \eqref{equ on fib of Y}. 
As $E_N/E_{N-1}$ is $\TT$-fixed, it is set theoretically supported on $S_0=\pi^{-1}(0)$ (proof of Proposition \ref{prop:properfixloc}).
As $E_N/E_{N-1}$ is stable, it is scheme theoretically supported on $S_0=\pi^{-1}(0)$ (e.g.~\cite[Lem.~2.2]{CMT1}).
By \cite[Lem.~2.4]{BG2}, we know 
\begin{equation*} \chi(\oO_C)\geqslant 1 \end{equation*}
and equality holds exactly when $[C]$ corresponds to a positive root.

Finally, from \eqref{equ1} and the claim we know $\chi(F)\geqslant 1$, and equality happens exactly
when $F=\oO_C$ is the structure sheaf of the CM curve $C$ corresponding to a positive root. 
\end{proof}

\begin{rmk}
In the proof of this lemma, we saw that the Cohen-Macaulay curves $C$ in part (2) have the property that $\mathcal{O}_C$ is stable, hence simple. In particular, $C$ is connected and $h^0(\mathcal{O}_C) = 1$, hence $h^1(\mathcal{O}_C)=0$, which will be used frequently in the remainder of this section.
\end{rmk}

The following lemma identifies deformation and obstruction spaces. 
\begin{lem} \label{lem:compareExts}
Let $\beta$ correspond to a positive root and let $I_C\cong(\oO_X\to \oO_C)\in P_1(X, \beta)^{\TT}$ be the unique fixed point as in Lemma \ref{lemma: fixed locus stable pairs}. 
We have $\TT$-equivariant isomorphisms 
\begin{align*}
    \Ext^i_X(I_C, I_C)_0&\cong \Ext^i_X(\oO_C, \oO_C), \quad i=1,2,3,\\
    \Hom_X(I_C, I_C)_0&=\Ext^4_X(I_C, I_C)_0=0.
\end{align*}
\end{lem}
\begin{proof}
Since $h^0(\oO_C)=1$, $h^1(\oO_C)=0$, the proof follows from the proof of \cite[Prop.~3.7]{Cao}.
\end{proof}
For a curve $C$ in class $\beta$ corresponding to a positive root, we have $h^0(\oO_C)=1$, $h^1(\oO_C)=0$. 
Hence the tautological complex $\oO_X^{[n]}$ restricts to 
\begin{equation} \label{eqn:tautrestr}
(\oO_X^{[n]}\otimes y^{-1})|_{I_C}\cong \BC\otimes y^{-1}.
\end{equation}
The computation of the stable pair invariants $P_{1,\beta}(X,\oO_X, y)$ reduces to the study of the weight decomposition of the torus representations $\Ext^i_X(\oO_C, \oO_C)$ for $i=1,2$. 

\begin{lem}\label{lemma: Ext of fourfolds and 3-folds}
Let $X=Y\times \BC$, where $Y$ is the Nakamura $G$-Hilbert scheme for $G<\mathrm{SO}(3)$, and let $C\subset Y$ be the Cohen-Macaulay curve in the class corresponding to a positive root. Then
\begin{align*}
    \Ext^1_X(\oO_C, \oO_C)&\cong \Ext^1_Y(\oO_C, \oO_C)\oplus \BC\otimes t_4^{-1},\\
        \Ext^2_X(\oO_C, \oO_C)&\cong \Ext^2_Y(\oO_C, \oO_C)\oplus  \Ext^2_Y(\oO_C, \oO_C)^*.
\end{align*}
Let $X=S\times \BC^2$, where $S$ is the Nakamura $G$-Hilbert scheme for $G<\mathrm{SU}(2)$, and let $C\subset S$ be the Cohen-Macaulay curve in the class corresponding to a positive root. Then
\begin{align*}
    \Ext^1_X(\oO_C, \oO_C)&\cong 
    \BC\otimes t_3^{-1}\oplus \BC\otimes t_4^{-1},\\
        \Ext^2_X(\oO_C, \oO_C)&\cong 
        \BC\otimes t_3^{-1}t_4^{-1}\oplus  \BC\otimes t_3t_4.
\end{align*}
\end{lem}
\begin{proof}
Let $Y\hookrightarrow X$ be the inclusion of the zero section. By adjunction in the derived category, we have
\[\mathbf{R}\Hom_X(\oO_C, \oO_C)=\mathbf{R}\Hom_Y(\oO_C, \oO_C)\oplus \mathbf{R}\Hom_Y(\oO_C, \oO_C)[-1]\otimes t_4^{-1}.\]
Taking cohomology we obtain
\begin{align*}
    \Ext^1_X(\oO_C, \oO_C)&\cong \Ext^1_Y(\oO_C, \oO_C)\oplus \Hom_Y(\oO_C, \oO_C)\otimes t_4^{-1}
    \\&\cong \Ext^1_Y(\oO_C, \oO_C)\oplus \BC\otimes t_4^{-1},\\
        \Ext^2_X(\oO_C, \oO_C)&\cong \Ext^2_Y(\oO_C, \oO_C)\oplus  \Ext^1_Y(\oO_C, \oO_C)\otimes t_4^{-1}\\
        &\cong \Ext^2_Y(\oO_C, \oO_C)\oplus  \Ext^2_Y(\oO_C, \oO_C)^*,
\end{align*}
where in the last line we used $\TT$-equivariant Serre duality.\footnote{Although $Y$ is non-compact, we can still use $\TT$-equivariant Serre duality because $C$ has proper support. Specifically, taking a $\TT$-equivariant compactification $j : Y \hookrightarrow \overline{Y}$, we have $\TT$-equivariant isomorphisms $\Ext^i_Y(\oO_C, \oO_C) \cong \Ext^i_{\overline{Y}}(j_!\oO_C, j_!\oO_C) \cong \Ext^{3-i}_{\overline{Y}}(j_!\oO_C, (j_!\oO_C) \otimes K_{\overline{Y}})^* \cong \Ext^{3-i}_{Y}(\oO_C, \oO_C)^*$.} 

The second claim follows similarly by applying adjunction to the zero section $S\hookrightarrow X$, taking cohomology, exploiting equivariant Serre duality on $S$, and using the fact that $\Ext^1_S(\oO_C, \oO_C)=0$. In fact, by Hirzebruch-Riemann-Roch
\begin{align*}
    \dim \Ext^1_S(\oO_C, \oO_C)&=2+C^2=0,
\end{align*}
where we used that $C^2=-2$ for all curve classes corresponding to positive roots.
\end{proof}
We can explicitly determine the invariants $P_{1,\beta}(X,\oO_X, y)$.
\begin{prop}\label{prop: positive root SU(2)} Let $X=S\times \BC^2$, where $S\to \BC^2/G$ is the minimal resolution of singularities for a finite group $G< \mathrm{SU}(2)$. 
If $\beta$ corresponds to a positive root, then we have
\[P_{1,\beta}(X,\oO_X, y)= \pm \frac{[t_3 t_4][y]}{[t_3][t_4]}. \]
Otherwise, $P_{1,\beta}(X,\oO_X, y)=0$.
\end{prop}
\begin{proof}
By Lemma \ref{lemma: fixed locus stable pairs}, invariants are zero if $\beta$ does not correspond to a positive root. If $\beta$ corresponds to a positive root, by Lemmata \ref{lem:compareExts} and \ref{lemma: Ext of fourfolds and 3-folds} 
\[\mathsf{V}:=t_3+t_4-t_3t_4\]
gives a square root of the virtual tangent space at $[I_C]\in P_{1}(X, \beta)^{\TT}$. Therefore the stable pair invariant is
\begin{equation*}
P_{1,\beta}(X,\oO_X, y)=\pm [-\mathsf{V}+y]=\pm \frac{[t_3 t_4][y]}{[t_3][t_4]}. \qedhere
\end{equation*}
\end{proof}

The finite abelian subgroups of $\mathrm{SU}(2)$ are the cyclic groups $\mathbb{Z}_r$. The finite abelian subgroups of $\mathrm{SO}(3)$ are the cyclic groups $\mathbb{Z}_r$ and $\mathbb{Z}_2 \times \mathbb{Z}_2$. 
In the case $\mathbb{Z}_2 \times \mathbb{Z}_2 < \mathrm{SO}(3)$, $f : S \to S_0 \subset Y$ contracts a single $\mathbb{P}^1$. The result is a configuration of three $\mathbb{P}^1$'s meeting mutually transversally in a point (i.e., locally, like the three coordinate axes of $\BC^3$) and each with normal bundle \cite[Rem.~25]{BG1}
$$
N_{\mathbb{P}^1/Y} \cong \mathcal{O}_{\mathbb{P}^1}(-1) \oplus \mathcal{O}_{\mathbb{P}^1}(-1).
$$
We denote the corresponding curve classes in $H_2(Y,\mathbb{Z})$ by $\beta_{01}, \beta_{10}, \beta_{11}$ (this notation is motivated by the next section).
\begin{rmk} \label{rmk: noZrinSO(3)}
  Note that any cyclic group $\mathbb{Z}_r < \mathrm{SO}(3)<\mathrm{SU}(3)$ is unitarily conjugate to a cyclic group $\mathbb{Z}_r < \mathrm{SU}(2) < \mathrm{SU}(3)$. Hence, we do not need to consider the cyclic subgroups of $\mathrm{SO}(3)$ separately. 
\end{rmk}
\begin{prop} 
Let $X=Y\times \BC$, where $Y$ is the Nakamura $G$-Hilbert scheme for $G< \mathrm{SO}(3)$. 
If $G=\mathbb{Z}_2\times \mathbb{Z}_2$, then on the curve classes corresponding to positive roots we have
\begin{align*}
& P_{1,\beta_{01}}(X,\oO_X, y)=P_{1,\beta_{10}}(X,\oO_X, y)=P_{1,\beta_{11}}(X,\oO_X, y)=  P_{1,\beta_{01}+\beta_{10}+\beta_{11}}(X,\oO_X, y)= \pm \frac{[y]}{[t_4]},\\
 &P_{1,\beta_{01}+\beta_{10}}(X,\oO_X, y)= \pm \frac{[t_1t_2t_3^{-1}][y]}{[t_3^2][t_4]}, \quad P_{1,\beta_{10}+\beta_{11}}(X,\oO_X, y)= \pm \frac{[t_1t_2^{-1}t_3][y]}{[t_2^2][t_4]}, \\
 &P_{1,\beta_{01}+\beta_{11}}(X,\oO_X, y)= \pm \frac{[t_1^{-1}t_2t_3][y]}{[t_1^2][t_4]}.
\end{align*}
\end{prop}
\begin{proof}
Since $G=\mathbb{Z}_2\times \mathbb{Z}_2$ is abelian, $X = Y \times \C$ is a toric  Calabi-Yau 4-fold. Hence the invariants can be calculated by a straight-forward application of the vertex/edge formalism of \cite{CKM, CK2}. In particular, by \cite[Prop.~2.6]{CK2}, the Zariski tangent space $\Ext^1_X(I_C,I_C)_0 \cong \Ext^1_X(\oO_C,\oO_C)$ has no $\TT$-fixed part and the equality in Lemma \ref{lemma: fixed locus stable pairs}(2) holds scheme-theoretically.
\end{proof}
\begin{rmk}
Let $X\to \BC^4/G$ be the crepant resolution of a finite subgroup $G< \mathrm{SO}(3)  < \mathrm{SU}(4)$ and let $\beta$ correspond to a positive root. When $G$ is abelian, $P_{n}(X, \beta)^\TT\cong \{(\oO_X\to \oO_C)\}$  holds scheme-theoretically, where $C$ is the curve in class $\beta$ \cite[Prop.~2.6]{CK2}. In the non-abelian case,  proving that the fixed point is reduced is equivalent to showing that $\Ext^1_X(\oO_C, \oO_C)^\TT=0$. Assuming this vanishing,  one can compute $P_{1, \beta}(X,\oO_X, y)$ by a similar analysis as in Proposition \ref{prop: positive root SU(2)}. In particular, in \cite[\S 3]{BG2}, Bryan-Gholampour have developed a method to determine the weight decomposition of $\Ext^1_Y(\oO_C, \oO_C)-\Ext^2_Y(\oO_C, \oO_C)$ for all finite $G< \mathrm{SO}(3)$. We illustrate this for the case $G=A_5<\mathrm{SO}(3)$ is the symmetry group of the icosahedron in Proposition \ref{prop: ico}. By applying the dimensional reduction and cohomological limit discussed in \S \ref{sec: Limit of theory}, we obtain invariants which coincide with the Gopakumar-Vafa invariants of crepant resolutions in \cite[Thm.~1.2]{BG2}. The invariants in loc.~cit.~were defined using moduli spaces of 1-dimensional stable sheaves.
\end{rmk}
\begin{prop} \label{prop: ico}
Let $X=Y\times \BC$, where $Y$ is the Nakamura $G$-Hilbert scheme for  $G=A_5<\mathrm{SO}(3)$  the symmetry group of the icosahedron, and assume that $ \Ext^1_Y(\oO_C, \oO_C)^\TT=0$ for $C$ the Cohen-Macaulay curve corresponding to a  positive root. Then on the curve classes corresponding to positive roots, we have
\begin{align*}
& P_{1,3\beta_1+5\beta_2+4\beta_3+3\beta_4}(X,\oO_X, y)= \pm \frac{[y]}{[t_4]}, \quad  P_{1,2\beta_1+4\beta_2+4\beta_3+2\beta_4}(X,\oO_X, y)= \pm \frac{[t][y]}{[t^2][t_4]} ,\\
 &P_{1,2\beta_1+4\beta_2+3\beta_3+2\beta_4}(X,\oO_X, y)= \pm \frac{[t^2][y]}{[t][t_4]} , \quad P_{1,\beta_1+2\beta_2+2\beta_3+\beta_4}(X,\oO_X, y)= \pm \frac{[t^2]^2[y]}{[t]^2[t_4]},
\end{align*}
where $\beta_1, \beta_2, \beta_3, \beta_4 $ are the classes of the irreducible components of the exceptional locus, using the notation of \cite[Prop.~3.1]{BG2}.
\end{prop}
\begin{proof}
    For $G=A_5$, by Lemmata \ref{lem:compareExts} and \ref{lemma: Ext of fourfolds and 3-folds} 
\[\mathsf{V}=\Ext^1_Y(\oO_C, \oO_C)-\Ext^2_Y(\oO_C, \oO_C)+t_4^{-1}\]
gives a square root of the virtual tangent space at $I_C\in P_{1}(X, \beta)^{\TT}$. The result follows from the weight decomposition computed in \cite[Prop.~3.1]{BG2}.
\end{proof}

For $n=0,1$, Conjecture \ref{conj: PT of G non-abelian} states 
$$
P_{n,\beta}(X,\mathcal{O}_X,y) = \left\{\begin{array}{cc} 0 & \textrm{if \ } n=0,  \\ P_{1,\beta}(X,\mathcal{O}_X,t_4) \frac{[y]}{[t_4]} &  \textrm{if \ } n=1.  \end{array}\right.
$$
We prove this prediction.
\begin{prop} \label{cor:n=0,1}
Conjecture \ref{conj: PT of G non-abelian} holds for $n=0,1$.
\end{prop}
\begin{proof}
By Lemma \ref{lemma: fixed locus stable pairs}, we only need to consider the case $n=1$ and $\beta$ corresponding to a positive root. Let $P:=P_{1}(X,\beta)$ and consider the virtual class $[P^{\TT}]^{\vir} \in A_*(P^{\TT},\mathbb{Z}[\tfrac{1}{2}]) \cong A_*(\{I_C\},\mathbb{Z}[\tfrac{1}{2}]) \cong \mathbb{Z}[\tfrac{1}{2}]\cdot \{I_C\}$ (for any choice of orientation). By Lemma \ref{lem:compareExts}, we have $[P^{\TT}]^{\vir} \in A_v(P^{\TT},\mathbb{Z}[\tfrac{1}{2}])$ where
$$
v = \mathrm{dim}\Ext_X^{1}(I_C,I_C)_0^{\TT} - \frac{1}{2} \mathrm{dim}\Ext_X^{2}(I_C,I_C)_0^{\TT} = \mathrm{dim}\Ext_X^{1}(\oO_C,\oO_C)^{\TT} - \frac{1}{2} \mathrm{dim}\Ext_X^{2}(\oO_C,\oO_C)^{\TT}.
$$
In the case $G < \mathrm{SU}(2)$, this equals zero by Lemma \ref{lemma: Ext of fourfolds and 3-folds}. In the case $G < \mathrm{SO}(3)$ this also equals zero by Lemma \ref{lemma: Ext of fourfolds and 3-folds} and \cite[Sect.~3]{BG2}. In either case, $[P^{\TT}]^{\mathrm{vir}} = \mu \cdot \{I_C\}$ for some $\mu \in \mathbb{Z}[\tfrac{1}{2}]$. By \cite{OT} and \eqref{eqn:tautrestr}, we have 
$$
\chi\left(P^\TT,\frac{\widehat{\oO}^{\vir}_{P^\TT}}{\sqrt{\mathfrak{e}^{\TT}}(N^{\vir})} \otimes \widehat{\Lambda}^{\mdot} (\oO_X^{[1]}|_{P^\TT} \otimes y^{-1}) \right) = \int_{\mu \cdot \{I_C\}}  \frac{\mathrm{ch}(\widehat{\Lambda}^{\mdot} (\mathbb{C} \otimes y^{-1}))}{\sqrt{e}(N^{\mathrm{vir}}| _{I_C})} \sqrt{\mathrm{td}}(T_{P}^{\mathrm{vir}}|_{I_C}),
$$
where $\mathrm{ch}(\widehat{\Lambda}^{\mdot} (\mathbb{C} \otimes y^{-1})) = [y]$ and no other term on the right-hand side depends on $y$. 
\end{proof}

\subsubsection{Vertex formalism}

In \cite{CKM}, we developed a vertex/edge formalism for computing stable pair invariants of toric Calabi-Yau 4-folds. This allows us to prove Conjecture \ref{conj: PT of G non-abelian} for $G$ abelian in many more cases.

\begin{prop}\label{prop: vertex formalism Z_r}
Let $G=\BZ_{r}<\mathrm{SU}(2)$. Then Conjecture \ref{conj: PT of G non-abelian} holds in the following cases.
\begin{itemize}
    \item For $n=0,1$ and all $\beta$.
    \item For $\beta$ is irreducible and all $n \geqslant 0$, $r\geqslant  2$.
    \item For $r=2$ and 
\begin{itemize}
\item $\beta=2\beta_1$ modulo $q^6$, 
\item $\beta=3\beta_1$ modulo $q^6$, 
\item $\beta=4\beta_1$ modulo $q^6$. 
\end{itemize}
\item For $r=3$ and
\begin{itemize}
\item $\beta = d \beta_1, d \beta_2$ with $d = 2,3,4$ modulo $q^6$, 
\item $\beta = \beta_1 + \beta_2$ modulo $q^6$, 
\item $\beta = \beta_1 + 2\beta_2, 2\beta_1 + \beta_2$ modulo $q^5$, 
\item $\beta = \beta_1 + 3\beta_2, 3\beta_1 + \beta_2$ modulo $q^5$, 
\item $\beta = 2\beta_1 + 2\beta_2$ modulo $q^4$.
\end{itemize}
\end{itemize}
Here $\beta_i$, for $i=1, \ldots, r-1$, are the curve classes of irreducible components of the exceptional divisor of the minimal resolution $S \to \mathbb{C}^2/G$. 
\end{prop}
\begin{proof}
The cases $n=0,1$ were covered in Corollary \ref{cor:n=0,1}.

For $\beta$ irreducible, Conjecture \ref{conj: PT of G non-abelian} (in general) states
$$
\mathcal{Z}^{\PT}_{X,\oO_X,\beta}(y,q) := \sum_{n} P_{n,\beta}(X,\mathcal{O}_X,y) q^n = - \frac{P_{1, \beta}(X,\oO_X,t_4) [y]}{[t_4][y^{\frac{1}{2}} q] [y^{\frac{1}{2}} q^{-1}]}.
$$
Let $\beta=\beta_i$. We proceed as in the proof of \cite[Prop.~B.2]{CKM}. We recall that, for appropriate choices of signs, the following expression holds for the vertex \cite[Lem.~B.1]{CKM}
$$
\mathsf{V}^{\PT}_{(1)\varnothing\varnothing\varnothing}(t,y,q) = \mathrm{Exp}\left( \frac{[yt_1]}{[t_1]} q \right).
$$
The normal bundle to the curve representing $\beta_i$ is $N_{\mathbb{P}^1/X}\cong \oO_{\mathbb{P}^1}(-2)\oplus\oO_{\mathbb{P}^1}\oplus \oO_{\mathbb{P}^1}$, therefore the edge term is \cite[\S 1]{CKM}
\[\mathsf{\tilde{e}}=t_3^{-1}- t_1^{-1}t_2^{-1}+t_1^{-1}t_2^{-1}t_3^{-1} -y. \]
Hence we conclude that, for appropriate choices of signs, the generating series is given by
\begin{align*}
\mathcal{Z}^{\PT}_{X,\oO_X,\beta}(y,q)&=\mathsf{V}^{\PT}_{(1)\varnothing\varnothing\varnothing}(t,y,q)|_{t_1=t_1^{-1}t_2,t_2=t_1^2} \cdot \mathsf{V}^{\PT}_{(1)\varnothing\varnothing\varnothing}(t,y,q)|_{t_1=t_1t_2^{-1},t_2=t_2^2} \\
&\, \quad \cdot q \cdot (-1) \cdot [-\tilde{\mathsf{e}}|_{t_1=t_1^{-1}t_2,t_2=t_1^2}] \\
&= -\frac{[y][t_1 t_2]}{[t_3][t_4]}q \, \Exp\left((y^{\frac{1}{2}}+y^{-\frac{1}{2}})q\right)\\
&=\frac{[t_1 t_2]}{[t_3]} \frac{[y]}{[t_4][y^{\frac{1}{2}} q] [y^{\frac{1}{2}} q^{-1}]}.  
\end{align*}
The second statement follows.

All other cases are established by implementing the vertex/edge formalism of \cite{CKM} into a Maple/Mathematica program.
\end{proof}

\begin{rmk}
The verifications of bullet points 2--4 of the previous proposition hold for certain choices of signs, i.e.~orientations, attached to the $\TT$-fixed points $M^{\TT}$. We expect that these orientations are induced from ``natural'' \emph{global} orientations on $M$. This remark holds analogously for the verifications in Propositions \ref{prop: vertex formalism Z_2 Z_2} and \ref{prop: r=1} below.
\end{rmk}

Using the same techniques, we prove the following for $G=\BZ_2\times \BZ_2$.
\begin{prop}\label{prop: vertex formalism Z_2 Z_2}
Let $G=\BZ_{2}\times \BZ_2<\mathrm{SO}(3)$. Then Conjecture \ref{conj: PT of G non-abelian} holds in the following cases.
\begin{itemize}
    \item For $n=0,1$ and all $\beta$.
    \item For $\beta$ irreducible.
    \item For $\beta=d \beta_{01}, d \beta_{10}, d \beta_{11}$  and
    \begin{itemize}
        \item $d=2$ modulo $q^6$,
        \item $d=3$ modulo $q^6$,
        \item $d=4$ modulo $q^7$.
    \end{itemize}
    \end{itemize}
\end{prop}

\section{Donaldson-Thomas theory of Calabi-Yau 4-orbifolds}

\subsection{Calabi-Yau 4-orbifolds}\label{sec: CY 4 orbifolds}

Following the notation of \cite[\S 2]{BCY}, we define:
\begin{defi}\label{defi of cy4 orb}
A \emph{Calabi-Yau 4-orbifold} is a smooth 4-dimensional quasi-projective Deligne-Mumford stack $\xX$ over $\BC$ having generically trivial stabilizers and trivial canonical bundle $K_{\xX}\cong \oO_{\xX}$. 
\end{defi}
Then the local model at a point $p\in \xX$ is $[\BC^4/G_{p}]$, where $G_p<\mathrm{SU}(4)$ is the finite group of automorphism at $p$.

We denote by $ K(\xX)$ the $K$-group of compactly supported coherent sheaves on $\xX$ modulo numerical equivalence, which comes with a natural filtration
\[F_0K(\xX)\subset F_1K(\xX)\subset F_2K(\xX)\subset F_3 K(\xX)\subset K(\xX),\]
where an element of $F_d K(\xX)$ can be represented by a formal sum of sheaves supported in dimension $d$ or less.

\subsection{Hilbert schemes on orbifolds}

Given $\alpha\in F_0K(\xX)$, we denote by $\Hilb^{\alpha}(\xX)$ the moduli stack parametrizing 0-dimensional closed substacks $\wW \subset \xX$ with class $[\oO_\wW]=\alpha$. 
Olsson-Starr \cite{OS} proved that $\Hilb^{\alpha}(\xX)$ is represented by an algebraic space as its objects do not have automorphisms. 
In fact, $\Hilb^{\alpha}(\xX)$ is a quasi-projective scheme by \cite[Thm.~1.5]{OS} (see also \cite{Nir}).

\begin{example}\label{ex: Hilb of quotient stack}
Let $\xX=[\BC^4/G]$, where $G < \mathrm{SU}(4)$
 is a finite group acting on $\BC^4$ by matrix multiplication. The $K$-group of $\xX$ is canonically identified with the representation ring
 \[
 K(\xX)\cong \mathbb{Z}[G^*],
 \]
 where $G^*$ denotes the character group of $G$. For any $G$-representation $R\in K(\xX)$, the $\mathbb{C}$-points of the  corresponding Hilbert scheme are given by
 \begin{align*}
    \Hilb^R([\BC^4/G])= \Big\{W\in\Hilb^{\dim R}(\mathbb{C}^4)\,|\,W \subset \mathbb{C}^4\textrm{ }\textrm{is 0-dim.}\textrm{ }G\textrm{-}\textrm{invariant}
\textrm{ }\textrm{and}\textrm{ }H^0(W,\oO_W)\cong R\Big\},  \nonumber
\end{align*}
where $H^0(W,\oO_W)\cong R$ is an isomorphism of $G$-representations.
Equivalently, there is a natural induced $G$-action on $\Hilb^n(\BC^4)$ and its $G$-fixed locus decomposes scheme-theoretically as
\begin{align*}
    \Hilb^n(\BC^4)^G=\bigsqcup_{\dim R=n} \Hilb^R([\BC^4/G]),
\end{align*}
where the disjoint union is over all (isomorphism classes of) $G$-representations of dimension $n$.
\end{example}

\subsection{Virtual structures}

Let $\xX$ be a Calabi-Yau 4-orbifold and $\alpha\in F_0K(\xX)$. By \cite[Cor.~2.13]{PTVV}, the moduli space $I:=\Hilb^\alpha(\xX)$ has a derived enhancement carrying a $(-2)$-shifted symplectic structure, which by \cite[Prop.~1.2]{STV} induces a 3-term symmetric obstruction theory 
\begin{align}\label{eqn: ot orbifold}
    \BE_\xX:=\mathbf{R}\pi_{I*}\mathbf{R}\HOM(\iI, \iI)_0[3]\to \mathbb{L}_{I},
\end{align}
where $\pi_I: \xX \times I \to I$ is the projection and  $\iI$ denotes the  ideal sheaf of the universal substack $\zZ\subset \xX\times I$.

We assume from now on that the obstruction theory \eqref{eqn: ot orbifold} is \emph{orientable}, as defined by \cite{CGJ}. Then there exist \emph{virtual fundamental classes} and  \emph{twisted virtual structure sheaves}
\begin{align*}
    [I]^{\vir} &\in\, A_{*}\left(I,\mathbb{Z}\left[\tfrac{1}{2}\right]\right), \\
    \widehat{\oO}_{I}^{\vir} &\in\, K_0\left(I, \mathbb{Z}\left[\tfrac{1}{2}\right]\right).
\end{align*}

\subsection{Orbifold Nekrasov genus}

Let $\xX$ be a  Calabi-Yau 4-orbifold,  $L$ a line bundle on $\xX$, and $\alpha\in F_0K(\xX)$. Assume moreover that the  obstruction theory \eqref{eqn: ot orbifold} is \emph{orientable}. We define the following \emph{tautological vector bundle} on $I:=\Hilb^\alpha(\xX)$
\begin{equation*} \label{eqn: taut orbi}
L^{[\alpha]} := \pi_{I*} (\pi_\xX^* L \otimes \oO_{\mathcal{Z}}),
\end{equation*}
where $\mathcal{Z}\subset \xX\times I$ is the universal closed substack and $\pi_\xX, \pi_I$ are the natural projections. Suppose $\xX=[X/G]$ is the global quotient stack of a Calabi-Yau 4-fold $X$ by the action of a finite group $G$ preserving the Calabi-Yau volume form. Let $\wW\in I$ be a point corresponding to a 0-dimensional $G$-invariant subscheme $W\subset X$. Then the fibre over $\wW$ satisfies
\[
L^{[\alpha]}|_{\wW}=H^0(\wW, \oO_\wW)=H^0(W, \oO_W)^G.
\]
\begin{defi} 
Let $\xX$ be a projective Calabi-Yau 4-orbifold. 
Assume $I:=\Hilb^\alpha(\xX)$ is orientable and fix an orientation. We define the \emph{orbifold Nekrasov genus} by
\begin{align*}
I_{\alpha}(\xX,L, y) &:= \chi\left(I, \widehat{\oO}_I^{\vir} \otimes \widehat{\Lambda}^{\mdot} (L^{[\alpha]} \otimes y^{-1}) \right)\in \BQ(y^{\frac{1}{2}}).
\end{align*}
Let $\xX$ be a (not necessarily projective) Calabi-Yau 4-orbifold with action of an algebraic torus $\TT$ preserving the Calabi-Yau volume form. 
Assume there exists a smooth projective 4-orbifold $\overline{\xX}$ with $\TT$-action and a $\TT$-equivariant open immersion $\xX \hookrightarrow \overline{\xX}$.\footnote{This assumption is satisfied in all the examples considered in this paper.}
Assume the fixed locus $I^\TT$ is proper and orientable, and fix an orientation.\footnote{Both assumptions are satisfied when $I^\TT$ is 0-dimensional reduced, e.g., as in Lemma \ref{lemma: fixed locus stack}.} Let $L$ be a $\TT$-equivariant line bundle on $\xX$. Then we define the orbifold Nekrasov genus by the virtual localization formula
\begin{align}\label{eqn: loc inv orbifold}
I_{\alpha}(\xX,L, y) &:= \chi\left(I^\TT,\frac{\widehat{\oO}_{I^\TT}^{\vir}}{\sqrt{\mathfrak{e}^{\TT}}(N^{\vir})} \otimes \widehat{\Lambda}^{\mdot} (L^{[\alpha]}|_{I^\TT} \otimes y^{-1}) \right)\in K_0^{\TT}(\pt)_{\mathrm{loc}}(y^{\frac{1}{2}}),
\end{align}
where $K^0_{\TT}(\pt)_{\mathrm{loc}}$ is the localization of $K^0_{\TT}(\pt)$ given by \eqref{equ for k0 loc}.
\end{defi}
We define the $K$-theoretic orbifold DT \emph{partition function} by
\begin{align}\label{K-orbi DT}
\mathcal{Z}^{\DT}_{\xX,L}(y,q)&:=\sum_{\alpha\in F_0K(\xX)}I_{\alpha}(\xX,L, y)\,q^\alpha\in K^0_{\TT}(\pt)_{\mathrm{loc}}(y^{\frac{1}{2}})(\!(q)\!),
 \end{align}
where  $q$ is a multi-index variable.

\subsection{Toric Calabi-Yau 4-orbifolds}

Building on Definition \ref{defi of cy4 orb}, we consider the following:
\begin{defi}
A \emph{toric} Calabi-Yau 4-orbifold $\xX$ is a 4-dimensional smooth quasi-projective toric Deligne-Mumford stack over $\BC$ with generically trivial stabilizers and  trivial canonical bundle $K_{\xX}\cong \oO_{\xX}$. 
\end{defi}
For the general theory of toric Deligne-Mumford stacks, we refer to \cite{BCS, FMN}.

As observed by Bryan-Cadman-Young, a toric Calabi-Yau orbifold $\xX$ (of any dimension) is uniquely determined by its coarse moduli space $X$, which is a toric variety in the usual sense \cite[Lem.~40]{BCY}. Therefore, a toric Calabi-Yau 4-orbifold is determined by a 4-dimensional toric variety with Gorenstein finite quotient singularities and trivial canonical bundle. As usual, the latter are determined by a lattice and fan $(N,\Sigma)$, and the trivial canonical divisor provides a linear function $\ell : N \to \mathbb{Z}$ such that $\ell(v_i) = 1$ for all generators $v_i$ of the 1-dimensional cones of $\Sigma$. In turn, the latter data can be translated into a so-called \emph{web diagram} (essentially the dual graph to the triangulation of the hyperplane $\{\ell = 1\}$ \cite[App.~B]{BCY}). The upshot is that, at each torus fixed point $p\in \xX$, there exists an open neighborhood of $\xX$ isomorphic to a global quotient stack $[\BC^4/G_p]$, where the finite abelian subgroup $G_p<(\BC^*)^4\cap \mathrm{SU}(4)$ on $\BC^4$ can be explicitly reconstructed from the web diagram \cite[Lem.~46]{BCY} (here $(\BC^*)^4$ denotes the maximal torus of $\mathrm{GL}(4,\mathbb{C})$).\footnote{Strictly speaking, \cite[App.~B]{BCY} is written for Calabi-Yau 3-orbifolds, but the arguments trivially generalize to Calabi-Yau 4-orbifolds.}

\subsection{Combinatorial description of the fixed locus}

Let $\xX$ be a toric Calabi-Yau 4-orbifold,  $\alpha\in F_0K(\xX)$ and let $\{p_j \}_j\in \xX$ the $(\BC^*)^4$-fixed points. We denote by 
\begin{equation}\label{equ on cy torus}\TT=\{t_1t_2t_3t_4=1\}\subset (\BC^*)^4 \end{equation}
the subtorus preserving the Calabi-Yau volume form of $\xX$. The  $\TT$-fixed locus of the Hilbert scheme  $\Hilb^\alpha(\xX)$ is described in terms of \emph{coloured solid partitions}, as we now explain.

At each fixed point $p_j\in \xX$, there exists an open toric chart of the form $[\BC^4/G_j]$, for a finite $G_j < (\mathbb{C}^*)^4 \cap  \mathrm{SU}(4)$, with its induced $\TT$-action. Since each element $\alpha\in F_0K(\xX)$ can be written as a formal sum of sheaves supported on the fixed points $\{p_j\}_j$, we can write
\[\alpha=\sum_j [\oO_{p_j}]\otimes R^{(j)}\in F_0K(\xX),\]
where each $R^{(j)}$ is a $G_j$-representation. In general this decomposition is non-unique. Whenever clear from the context, we simply identify $[\oO_{p_j}]\otimes R^{(j)}$ with its $G_j$-representation $R^{(j)}\in \mathbb{Z}[G_j^*]$, where $G_j^*$ denotes the group of characters of $G_j$.
\begin{lem}\label{lemma: fixed locus stack}
There exists an isomorphism $\TT$-fixed loci
\[
\Hilb^{\alpha}(\xX)^{\TT}\cong \coprod_{\alpha=\sum_j R^{(j)}}\prod_{j}\Hilb^{R^{(j)}}([\BC^4/G_j])^{\TT},
\]
where the disjoint union is over all possible decompositions $ \alpha=\sum_j R^{(j)}$. Moreover, the $\TT$-fixed loci are reduced and 0-dimensional.
\end{lem}
\begin{proof}
Denote by $\uU_j=[\BC^4/G_j]$ the toric charts at the torus fixed points $p_j\in \xX$ as above. Clearly, restriction to $\uU_j$ provides a morphism
\begin{equation} \label{eqn: restrmor}
\Hilb^{\alpha}(\xX)^{\TT}\to \coprod_{\alpha=\sum_j R^{(j)}}\prod_{j}\Hilb^{R^{(j)}}([\BC^4/G_j])^{\TT}.
\end{equation}
The $\mathbb{C}$-points of $\Hilb^{\alpha}(\xX)^{\TT}$ are 0-dimensional coherent sheaves supported at the torus fixed points $p_j$, from which it follows that \eqref{eqn: restrmor} is a bijection on $\mathbb{C}$-points. Furthermore
$$
\Hilb^{R^{(j)}}([\BC^4/G_j])^{\TT} \hookrightarrow  \Hilb^n(\BC^4)^\TT,
$$
where $\dim R^{(j)}=n$. Since $\Hilb^n(\BC^4)^\TT$ is 0-dimensional reduced \cite[Lem.~3.6]{CK1}, we find that the codomain of \eqref{eqn: restrmor} is 0-dimensional reduced. It remains to show that $\Hilb^{\alpha}(\xX)^{\TT}$ is reduced. 

For any $\mathbb{C}$-point $\wW \in \Hilb^{\alpha}(\xX)^{\TT}$, we can write
$
\oO_\wW = \bigoplus_j \oO_{\wW_j}
$
where $\wW_j = \wW|_{\uU_j}$ is a 0-dimensional closed substack supported at $p_j$. We determine the tangent space $\Hom_{\xX}(I_\wW,\oO_\wW) \cong \Ext^1_{\xX}(I_\wW,I_\wW)_0 \cong \Ext^1_{\xX}(\oO_\wW,\oO_\wW)$ at $\wW$ (these isomorphisms follow as in the proof of Lemma \ref{lem:compareExts}). Since the $\wW_j$ are 0-dimensional with disjoint supports, we have
$$
\Ext^1_{\xX}(\oO_\wW,\oO_\wW) \cong \bigoplus_j \Ext^1_{\uU_j}(\oO_{\wW_j},\oO_{\wW_j}).
$$
By the first part of the proof, the right-hand side has $\TT$-fixed part equal to zero. Hence \eqref{eqn: restrmor} is an isomorphism on tangent spaces and the result follows. 
\end{proof}
By Lemma \ref{lemma: fixed locus stack} we just need to classify the $\TT$-fixed points in the local case of a global quotient stack $[\BC^4/G]$, where $G<(\mathbb{C}^*)^4 \cap \mathrm{SU}(4)$ is a finite abelian subgroup. By the construction in Example \ref{ex: Hilb of quotient stack}, such $\TT$-fixed points will form a subset of the $\TT$-fixed points of $\Hilb^n(\BC^4)$, whose description we now recall.
\begin{defi}
A \emph{solid partition} is a collection of finitely many lattice points $\pi\subset \mathbb{Z}^4_{\geq 0}$ such that if any of $(i+1,j,k,l),(i,j+1,k,l),(i,j,k+1,l),(i,j,k,l+1)$ is in $\pi$, then $(i,j,k,l) \in \pi$. We denote the \emph{size} of a solid partition by $|\pi|$, i.e. the number of boxes in $\pi$.
\end{defi}
By (for example) \cite[\S 3.1]{CK1}, the $\TT$-fixed locus of the ordinary Hilbert scheme of points $\Hilb^n(\BC^4)$ is described as follows
\begin{align*}
    \Hilb^n(\BC^4)^\TT=\left\{ \pi \mbox{ solid partition}: |\pi|=n\right\}.
\end{align*}
In the orbifold setting we need to distinguish which solid partitions correspond to $\TT$-fixed points in  $  \Hilb^R([\BC^4/G])^\TT\subset  \Hilb^n(\BC^4)^\TT$, where $\dim R=n$.
\begin{defi}[{\cite[Def.~1.2]{Young}}]
Let $G< (\mathbb{C}^*)^4 \cap  \mathrm{SU}(4)$ be any finite abelian group (where $(\mathbb{C}^*)^4 < \mathrm{GL}(4, \BC)$ denotes the diagonal subgroup).
A \emph{colouring} (with respect to $G$) is a homomorphism of additive monoids
\begin{align*}
    K:(\BZ_{\geqslant  0})^4\to G.
\end{align*}
\end{defi}
Note that a colouring is determined only by the images of $(1,0,0,0)$, $(0,1,0,0)$, $(0,0,1,0)$, $(0,0,0,1)$. We refer to the elements of $G$ as \emph{colours}; by the bijection $G \cong G^*$ they correspond to the irreducible representations of $G$. In particular, $K$ assigns an irreducible representation of $G$ to each element of $\mathbb{Z}_{\geq 0}^4$. 
\begin{defi}
\hfill
\begin{enumerate}
\item[(1)] Let $K : \mathbb{Z}_{\geq 0}^4 \to G$ be a colouring. Given a solid partition $\pi\subset (\BZ_{\geqslant  0})^4$, we say that a box $(i,j,k,l)\in \pi$ is \emph{$R$-coloured} if the image $K(i,j,k,l)$ corresponds to the irreducible $G$-representation $R$. 
\item[(2)] For $R\in G^*$, we denote by $|\pi|_R$ the number of boxes in $\pi$ of colour $R$.  
\item[(3)] For a $G$-representation $R=\sum_{i=0}^rd_i R_i$, where $R_0,\dots, R_r$ denote the irreducible $G$-representations, we say that a solid partition $\pi$ is $R$-\textit{coloured} if $|\pi|_{R_i}=d_i$ for all $i=0,\dots, r$. In this case, we write $|\pi|_G=R$.
\end{enumerate}
\end{defi}
Let $G< (\mathbb{C}^*)^4 \cap  \mathrm{SU}(4)$ be a finite abelian subgroup. Then $G$ defines a colouring via its standard action on $\mathbb{C}^4$ (defined by matrix multiplication). Denoting the standard coordinates of $\mathbb{C}^4$ by $x,y,z,w$, we obtain 1-dimensional vector spaces $R_x := (x)/(x^2,y,z,w)\cap (x)$, $R_y := (y)/(x,y^2,z,w)\cap(y)$, $R_z := (z)/(x,y,z^2,w)\cap(z)$, $R_w:= (w)/(x,y,z,w^2)\cap(w)$ which are naturally irreducible $G$-representations. 
We define the associated colouring function by
\[
K(i,j,k,l)=R_x^{\otimes i}\otimes R_y^{\otimes j} \otimes R_z^{\otimes k}\otimes R_w^{\otimes l}.
\]
Then we have an isomorphism of (infinite dimensional) $G$-representations
$$
\mathbb{C}[x,y,z,w] \cong \bigoplus_{(i,j,k,l) \in \mathbb{Z}^4_{\geq 0}} K(i,j,k,l).
$$
 Given a $G$-representation $R$, we have
\begin{align*}
    \Hilb^R([\BC^4/G])^\TT=\left\{ \pi \mbox{ solid partition}: |\pi|_{G}=R\right\}.
\end{align*}

\subsection{Orbifold vertex formalism}\label{sect on orb vertex}

Let $\xX$ be a toric Calabi-Yau 4-orbifold, $\alpha\in F_0K(\xX)$, and let $\TT\subset (\BC^*)^4$ be the Calabi-Yau subtorus \eqref{equ on cy torus}. As in \S \ref{sec: virtual localization}, the $\TT$-action lifts to $\Hilb^\alpha(\xX)$ making the obstruction theory $\BE_\xX$ \eqref{eqn: ot orbifold} $\TT$-equivariant. Given any $\wW\in \Hilb^\alpha(\xX)$, we denote the \emph{virtual tangent space} at $\wW$ by
$$T^{\vir}_{\xX, \wW}:=\BE^\vee_\xX|_{\wW} \in K^0_{\TT}(\pt).$$
\begin{lem}
For every $\TT$-fixed point $\wW\in \Hilb^\alpha(\xX)^\TT$, the virtual tangent space  $T^{\vir}_{\xX, \wW}$ is $\TT$-movable.
\end{lem}
\begin{proof}
Using the notation of Lemma \ref{lemma: fixed locus stack}, we have
$$
-\mathbf{R}\Hom_{\xX}(I_{\wW},I_{\wW})_0 = - \sum_j \mathbf{R}\Hom_{\mathcal{U}_j}(I_{\wW_j},I_{\wW_j})_0.
$$
This equality follows from the fact that $\Hom_{\xX}(I_{\wW},I_{\wW})_0 = \Ext^4_{\xX}(I_{\wW},I_{\wW})_0 = 0$, $\Ext^i_{\xX}(I_{\wW},I_{\wW})_0 \cong \Ext^i_{\xX}(\mathcal{O}_{\wW},\mathcal{O}_{\wW})$ for $i=1,2,3$, and 
$$\Ext^i_{\xX}(\mathcal{O}_{\wW},\mathcal{O}_{\wW}) \cong \bigoplus_j \Ext^i_{\mathcal{U}_j}(\mathcal{O}_{\wW_j},\mathcal{O}_{\wW_j})$$
for all $i=1,2,3$.
Therefore we may assume that $\xX=[\BC^4/G]$ for a finite abelian subgroup $G< (\mathbb{C}^*)^4 \cap  \mathrm{SU}(4)$ and we claim that $T^{\vir}_{[\BC^4/G],\wW}$ is $\TT$-movable for any $\wW\in  \Hilb^R([\BC^4/G])^\TT$, where $R$ is a $G$-representation of some dimension $n$. 
Let $W \subset \mathbb{C}^4$ be the $\TT$-fixed and $G$-fixed 0-dimensional closed subscheme corresponding to $\wW$.
Denote by $T^{\vir}_{\BC^4, W}$ the virtual tangent space at $W\in \Hilb^n(\BC^4)$.
Then
\begin{align*}
    T^{\vir}_{[\BC^4/G],\wW}=(T^{\vir}_{\BC^4, W})^{G},
\end{align*}
where by $(\cdot)^G$ we denote the $G$-fixed part when seeing $T^{\vir}_{\BC^4, W}$ as a $G$-representation. The conclusion follows from the fact that $T^{\vir}_{\BC^4, W}$ is $\TT$-movable  (cf.~\cite[Prop.~2.3]{Mon}, see also \cite[Lem.~2.2]{CK2}) and the fact that taking the $G$-fixed  part commutes with taking the  $\TT$-fixed part.
\end{proof}

\subsubsection{Ordinary vertex formalism} 

We recall the vertex formalism for Donaldson-Thomas invariants associated to $\Hilb^n(\mathbb{C}^4)$ developed by Nekrasov-Piazzalunga \cite{NP}. We will follow the conventions of \cite{CKM}
and then upgrade it to the orbifold setting.

Let $Z\in \Hilb^n(\BC^4)^\TT$ be the fixed point corresponding to a solid partition $\pi$ of size $n$. The $(\BC^*)^4$-representation of the space of global sections $H^0(\mathbb{C}^4,\oO_{Z})$ is given by
\begin{align*}
    Z_\pi:=\sum_{(i,j,k,l)\in \pi}t_1^i t_2^j t_3^k t_4^l \in K^0_{(\mathbb{C}^*)^4}(\pt) \cong \BZ[t_1^{\pm 1},\dots, t_4^{\pm 1}].
\end{align*}
Let $I_Z \subset \oO_{\mathbb{C}^4}$ be the ideal sheaf of $Z \subset \mathbb{C}^4$ and $\oO_Z$ its structure sheaf. Then we can express  the virtual tangent bundle $T^{\vir}_{\BC^4,Z}$  in the equivariant $K$-group as
\begin{align*}
    T^{\vir}_{\BC^4,Z}&=\mathbf{R}\Hom_{\mathbb{C}^4}(\oO_{\mathbb{C}^4},\oO_{\mathbb{C}^4})-\mathbf{R}\Hom_{\mathbb{C}^4}(I_Z,I_Z)\\
    &=\mathbf{R}\Hom_{\mathbb{C}^4}(\oO_{\mathbb{C}^4},\oO_Z)+\mathbf{R}\Hom_{\mathbb{C}^4}(\oO_Z,\oO_{\mathbb{C}^4})-\mathbf{R}\Hom_{\mathbb{C}^4}(\oO_Z,\oO_Z)\\
    &= Z_\pi+\frac{\overline{Z}_\pi}{t_1 t_2 t_3 t_4}-\frac{P_{1234}Z_\pi\overline{Z}_\pi}{t_1 t_2 t_3 t_4},
\end{align*}
where $P_{I}=\prod_{i\in I}(1-t_i)$ for any set of indices $I \subset \{1,2,3,4\}$ and $\overline{(\cdot)}$ denotes the involution on $K^0_{(\BC^*)^4}(\pt)$ sending $t_i \mapsto 1/t_i$ (see \cite[\S 3.4.3]{O} for a similar computation in 3 dimensions). If we restrict to the subtorus $\TT=\{t_1t_2t_3t_4=1\}\subset (\BC^*)^4$ preserving the Calabi-Yau form, then $T_{\BC^4,Z}^{\vir}$ admits the ``square root''
\begin{align*}
    T_{\BC^4,Z}^{\vir}=\mathsf{v}_\pi+\overline{\mathsf{v}}_\pi,
\end{align*}
where 
\begin{align*}
    \mathsf{v}_\pi:= Z_\pi- \overline{P}_{123}Z_\pi \overline{Z}_\pi\in K^0_\TT(\pt)\cong\frac{\BZ[t_1^{\pm 1},\dots, t_4^{\pm 1}]}{(t_1t_2t_3t_4-1)}.
\end{align*}
Following Nekrasov-Piazzalunga, we twist the vertex term $\mathsf{v}_\pi$ to include the $K$-theoretic insertion in the definition of Nekrasov genus
\begin{align*}
     \tilde{\mathsf{v}}_\pi:= Z_\pi-\overline{Z}_\pi\cdot y- \overline{P}_{123}Z_\pi \overline{Z}_\pi\in K^0_\TT(\pt)\cong\frac{\BZ[t_1^{\pm 1},\dots, t_4^{\pm 1}, y]}{(t_1t_2t_3t_4-1)}.
\end{align*}
Then by \eqref{eqn: localization with []} the invariants \eqref{eqn: inv loc ordinary}\footnote{For the general case where $\mathcal{O}_{\BC^4}$ is replaced by a $\TT$-equivariant line bundle $L$, one just needs to substitute $y$ by $y\cdot L^*|_Z$ at every fixed point $Z\in  \Hilb^n(\BC^4)^\TT$, where we identify $L^*|_Z$ with its character.} are computed as (cf. \cite[Thm. 1.13]{CKM})
\begin{align*}
    I_{n,0}(\BC^4,\oO_{\BC^4}, y)=\sum_{|\pi|=n} (-1)^{\sigma_\pi} [-\tilde{\mathsf{v}}_\pi].
\end{align*}
Here $(-1)^{\sigma_\pi}$ is a choice of sign at the fixed point $\pi$. It is proved by the second-named author and Rennemo \cite{KR} that there exist natural \emph{global} orientations on $\Hilb^n(\mathbb{C}^4)$ (and the virtual normal bundle at the fixed points) such that the induced orientation at a fixed point $\pi$ is given by
\begin{align*}
    \sigma_\pi=|\pi|+|\{(a,a,a,d)\in \pi: a<d\}|.
\end{align*}
This recovers the sign rule originally discovered by Nekrasov-Piazzalunga \cite{NP}.

\subsubsection{Orbifold vertex formalism}

Let $G<(\mathbb{C}^*)^4 \cap  \mathrm{SU}(4)$ be a finite abelian group, $R$ a $G$-representation, and $W\in \Hilb^R([\BC^4/G])^\TT\subset  \Hilb^{\dim R}(\BC^4)^\TT$ a $\TT$-fixed point corresponding to an $R$-coloured solid partition $\pi$. Then $T_{\BC^4, W}^{\vir}$ is naturally a $G\times \TT$-representation. At the level of fixed points, we have 
\begin{align*}
  T^{\vir}_{[\BC^4/G],W}=(T^{\vir}_{\BC^4,W})^{G}, 
\end{align*}
where $(\cdot)^G$ denotes the $G$-fixed part. Since the involution $\overline{(\cdot)}$ on $ K^0_\TT(\pt)$ commutes with taking the $G$-fixed part, we have
\begin{align*}
    (T_{\BC^4,W}^{\vir})^G=\mathsf{v}_\pi^G+\overline{\mathsf{v}_\pi^G},
\end{align*}
i.e.~$\mathsf{v}_\pi^G$ is a square root of the virtual tangent bundle $T^{\vir}_{[\BC^4/G],W}$ at $W\in \Hilb^R([\BC^4/G])^\TT$. 
Therefore  the invariants \eqref{eqn: loc inv orbifold}
are computed as 
\begin{align*}
    I_{R}(\BC^4,\oO_{\BC^4}, y)=\sum_{|\pi|_G=R} (-1)^{\sigma^G_\pi} [-\tilde{\mathsf{v}}^G_\pi],
\end{align*}
where the sum is over all $R$-coloured solid partitions and $ (-1)^{\sigma^G_\pi}$ is a choice of sign.\footnote{In all the computations performed in \S \ref{sec: orbifold partition functions}, the correct sign is $|\pi|_{R_0}+|\{(a,a,a,d)\in \pi: a<d\}|$, where $R_0$ is the trivial representation. We expect that these signs are induced from global orientations in a similar fashion to the case of $\Hilb^n(\BC^4)$ studied in \cite{KR}.}

Denote the irreducible representations of $G$ by $R_0,\dots, R_r$. The partition function of $K$-theoretic orbifold Donaldson-Thomas invariants takes the following explicit form
\begin{align*}
    \zZ^{\DT}_{[\BC^4/G], \oO}(y,q_0,\dots, q_r)=\sum_{\pi}(-1)^{\sigma^G_\pi}[-\tilde{\mathsf{v}}^G_\pi]\cdot q_0^{|\pi|_{R_0}}\cdots q_r^{|\pi|_{R_r}}.
\end{align*}
In the case of a general toric Calabi-Yau 4-orbifold, the invariants are easily computed through the orbifold vertex formalism by considering the localized contribution of each toric chart with suitable change of variables (cf.~\cite[Eqn.~(2.8)]{Mon}).

\subsection{Examples} \label{sec: orbivertexexamples}

We compute some examples using the orbifold vertex formalism.
\begin{example}
Let $\BZ_r$ act on $\BC^{4}$ by 
\begin{align*}
    \epsilon\cdot (x_1, x_2, x_3, x_4)=(\epsilon x_1, \epsilon^{-1} x_2, x_3, x_4),
\end{align*}
where $\epsilon=e^{2\pi \sqrt{-1} / r}$ is a primitive $r$-th root of unity and we denote by $R_1$ the irreducible $\BZ_r$-representation corresponding to $\epsilon$.
Given a solid partition $\pi$, each box $(i,j,k,l)\in \pi$ has colour $R_1^{i-j}$.

Let $W\in\Hilb^R([\BC^4/{\BZ_r}])^\TT $ be a fixed point corresponding to an $R$-coloured solid partition $\pi$. 
As a $(\BZ_r\times \TT)$-representation, the space of global sections $H^0(W,\oO_W)$ is given by
\begin{align*}
    W_\pi&:=\sum_{(i,j,k,l)\in \pi}R_1^{i-j}t_1^i t_2^j t_3^k t_4^l\\ 
    &=\sum_{l=0}^{r-1}W^{(l)}_{\pi}\cdot R_1^{l},
\end{align*}
where $R=\sum_{l=0}^{r-1}\dim W^{(l)}_{\pi}\cdot R_1^{l}$. An easy computation leads to 
\begin{multline*}
    \tilde{\mathsf{v}}^{\BZ_r}_\pi=W_\pi^{(0)}-\overline{W_\pi^{(0)}}y-(1-t_3^{-1})\left((1+t_1^{-1}t_2^{-1})\cdot\sum_{l+k=0\, \mathrm{mod}\, r}W_\pi^{(l)}\overline{W_\pi^{(k)}}\right.\\
    \left.-t_1^{-1}\cdot \sum_{l+k=1 \, \mathrm{mod}\, r}W_\pi^{(l)}\overline{W_\pi^{(k)}}-t_2^{-1}\cdot \sum_{l+k=-1 \, \mathrm{mod}\, r}W_\pi^{(l)}\overline{W_\pi^{(k)}}\right).
\end{multline*}
\end{example}
\begin{example}
Let $\BZ_2\times \BZ_2=\langle g_1,g_2\rangle $ act on $\BC^{4}$ by 
\begin{align*}
    g_1\cdot (x_1, x_2, x_3, x_4)=( x_1, - x_2, -x_3, x_4), \quad  g_2\cdot (x_1, x_2, x_3, x_4)=( -x_1,  x_2, -x_3, x_4),
\end{align*}
and denote by $R_{00}, R_{10}, R_{01}, R_{11}$ the induced irreducible $(\BZ_2\times \BZ_2)$-representations.
Given a solid partition $\pi$, we assign to each box $(i,j,k,l)\in \pi$ the colour $R_{10}^{i}R_{01}^{j}R_{11}^{k}$ with the obvious relations 
$$R_{10}R_{01}=R_{11}, \quad R_{ab}^2=R_{00}. $$
Let $W\in\Hilb^R([\BC^4/\BZ_2\times \BZ_2])^\TT $ be a fixed point corresponding to an $R$-coloured solid partition $\pi$. As a $(\BZ_2\times \BZ_2\times \TT)$-representation, the space of global sections $H^0(W,\oO_W)$ is given by
\begin{align*}
    W_\pi&:=\sum_{(i,j,k,l)\in \pi}R_{10}^{i}R_{01}^{j}R_{11}^{k} t_1^i t_2^j t_3^k t_4^l\\
    &=\sum_{0\leqslant  a,b\leqslant  1}W^{(ab)}_{\pi}\cdot R_{ab},
\end{align*}
where $R=\sum_{0\leqslant  a,b\leqslant  1}\dim W^{(ab)}_{\pi}\cdot R_{ab}$. An easy computation leads to 
\begin{multline*}
    \tilde{\mathsf{v}}^{\BZ_2\times \BZ_2}_\pi=W_\pi^{(00)}-\overline{W_\pi^{(00)}}y-(1-t_1^{-1}t_2^{-1}t_3^{-1})\left(\sum_{0\leqslant  a,b\leqslant  1}W_\pi^{(ab)}\overline{W_{\pi}^{(ab)}}\right) \\
    -(-t_1^{-1}+t_2^{-1}t_3^{-1})\left(\sum_{0\leqslant  a,b\leqslant  1}W_\pi^{(ab)}\overline{W_{\pi}^{(a-1,b)}} \right) 
    -(-t_2^{-1}+t_1^{-1}t_3^{-1})\left(\sum_{0\leqslant  a,b\leqslant  1}W_\pi^{(ab)}\overline{W_{\pi}^{(a,b-1)}} \right) \\
    -(-t_3^{-1}+t_1^{-1}t_2^{-1})\left(\sum_{0\leqslant  a,b\leqslant  1}W_\pi^{(ab)}\overline{W_{\pi}^{(a-1,b-1)}} \right),
\end{multline*}
where all superscript are considered modulo 2.
\end{example}
\begin{example}
Let $\BZ_3$ act on $\BC^{4}$ by 
\begin{align*}
    \epsilon\cdot (x_1, x_2, x_3, x_4)=(\epsilon x_1, \epsilon x_2, \epsilon x_3, x_4),
\end{align*}
where $\epsilon=e^{2\pi \sqrt{-1} /3}$ is a primitive root of unity and denote by $R_1$ the irreducible $\BZ_3$-representation corresponding to $\epsilon$.
Given a solid partition $\pi$, we assign to each box $(i,j,k,l)\in \pi$ the colour $R_1^{i+j+k}$.

Let $W\in\Hilb^R([\BC^4/{\BZ_3}])^\TT $ be a fixed point corresponding to an $R$-coloured solid partition $\pi$. 
As a $(\BZ_3\times \TT)$-representation, the space of global sections $H^0(W,\oO_W)$ is given by 
\begin{align*}
    W_\pi&=\sum_{(i,j,k,l)\in \pi}R_1^{i+j+k}t_1^i t_2^j t_3^k t_4^l\\
    &=\sum_{l=0}^{2}W^{(l)}_{\pi}\cdot R_1^{l},
\end{align*}
where $R=\sum_{l=0}^{2}\dim W^{(l)}_{\pi}\cdot R_1^{l}$. An easy computation leads to 
\begin{multline*}
    \tilde{\mathsf{v}}^{\BZ_3}_\pi=W_\pi^{(0)}-\overline{W_\pi^{(0)}}y-(1-t_1^{-1}-t_2^{-1}-t_3^{-1})\left(\sum_{l=0}^{2}W_\pi^{(l)}\overline{W_\pi^{(l)}}\right)\\
    +(t_1^{-1}+t_2^{-1}+t_3^{-1})\left(\sum_{l=0}^{2}W_\pi^{(l)}\overline{W_\pi^{(l-1)}}\right)
    -(t_1^{-1}t_2^{-1}+t_1^{-1}t_3^{-1}+t_2^{-1}t_3^{-1})\left(\sum_{l=0}^{2}W_\pi^{(l)}\overline{W_\pi^{(l+1)}}\right).
\end{multline*}
\end{example}

\subsection{Orbifold partition functions}\label{sec: orbifold partition functions}

Let $G < \mathrm{SU}(4)$ be a finite abelian subgroup with elements of age no greater than 1. By Proposition \ref{classify g=1 abelian} and Remark \ref{rmk: noZrinSO(3)}, the only cases to consider are $\mathbb{Z}_r < \mathrm{SU}(2) < \mathrm{SU}(4)$, where $r \geqslant 1$, and $\mathbb{Z}_2 \times \mathbb{Z}_2 < \mathrm{SO}(3) < \mathrm{SU}(4)$. 
We conjecture  closed expressions for the Nekrasov genera of $[\BC^4/\BZ_r]$ and $[\BC^4/\BZ_2\times \BZ_2]$. Our expressions simultaneously generalize Nekrasov's conjecture for $\mathbb{C}^4$ \cite{Nek} and Cirafici's conjectures \cite[Eqn.~(5.37),~(5.56)]{Cir} for $ [\BC^3/\BZ_r], [\BC^3/\BZ_2\times \BZ_2]$ in the physics literature.
\begin{conj}\label{conj: closed formula orbifold vertex}
There exist orientations such that
\begin{align*}
\mathcal{Z}^{\DT}_{[\mathbb{C}^4/\BZ_{r}],\oO}(y,q_0,\ldots,q_{r-1}) = \Exp\left(\mathcal{F}_r(t, y, q_{(r)})+\mathcal{F}^{\mathrm{col}}_{r}(t,y,q_0,\dots, q_{r-1})\right),
\end{align*}
where 
\begin{align*}
\mathcal{F}(t, y, q)&:=\frac{[t_1t_2][t_2t_3][t_1t_3]}{[t_1][t_2][t_3][t_4]}\cdot\frac{[y]}{[y^{\frac{1}{2}}q][y^{\frac{1}{2}}q^{-1}]},\\
    \mathcal{F}_r(t, y, q)&:=\sum_{k=0}^{r-1}\mathcal{F}(t_{1}^{r-k}t_2^{-k}, t_2^{k+1}t_1^{-r+k+1}, t_3, t_4, y, q),\\
    \mathcal{F}^{\mathrm{col}}_{r}(t,y,q_0,\dots, q_{r-1})&:=\sum_{0<i\leqslant  j<r}\left(q_{[i,j]}+q_{[i,j]}^{-1}\right)\frac{[t_1 t_2][y]}{[t_3][t_4][y^{\frac{1}{2}}q_{(r)}][y^{\frac{1}{2}}q_{(r)}^{-1}]},
\end{align*}
with  $q_{[i,j]}:=q_i\cdots q_j $ and $q_{(r)}:=q_0\cdots q_{r-1}$.
\end{conj}
\begin{conj}\label{conj: closed formula orbifold vertex D4}
There exist orientations such that
\begin{multline*}
\mathcal{Z}^{\DT}_{[\mathbb{C}^4/\BZ_{2}\times \BZ_2],\oO}(y,q_{00},q_{10}, q_{01}, q_{11}) =\Exp\left(\mathcal{F}_{2,2}(t, y, q_{(2,2)})+\mathcal{F}^{\mathrm{col}}_{2,2}(t,y,q_{00},q_{10},q_{01}, q_{11}) \right),
\end{multline*}
where 
\begin{multline*}
     \mathcal{F}_{2,2}(t_1,t_2,t_3,t_4, y, q):=\mathcal{F}\left(t_1^2, t_2^{2}, \frac{t_3}{t_1t_2}, t_4, y, q\right)+\mathcal{F}\left(t_1^2, \frac{t_2}{t_1t_3}, t_3^{2}, t_4, y, q\right)\\
    +\mathcal{F}\left(\frac{t_1}{ t_2t_3}, t_2^{2}, t_3^{2}, t_4, y, q\right)+\mathcal{F}\left(\frac{t_2t_3}{t_1}, \frac{t_1t_3}{t_2},\frac{t_1t_2}{t_3} , t_4, y, q\right),
\end{multline*}
\begin{multline*}
     \mathcal{F}^{\mathrm{col}}_{2,2}(t,y,q_{00},q_{10},q_{01}, q_{11}):=\frac{[y]}{[t_4][y^{\frac{1}{2}} q_{(2,2)}][y^{\frac{1}{2}} q_{(2,2)}^{-1}]}\Bigg(\frac{[t_1t_2t_3^{-1}]}{[t_3^{2}]}(q_{10}q_{01}+q_{10}^{-1}q_{01}^{-1})\\
     +\frac{[t_1t_2^{-1}t_3]}{[t_2^{2}]}(q_{10}q_{11}+q_{10}^{-1}q_{11}^{-1})+\frac{[t_1^{-1}t_2t_3]}{[t_1^{2}]}(q_{01}q_{11}+q_{01}^{-1}q_{11}^{-1})\\
     +q_{10}+q_{01}+q_{11}+q_{10}q_{01}q_{11}+q_{10}^{-1}+q_{01}^{-1}+q_{11}^{-1}+q_{10}^{-1}q_{01}^{-1}q_{11}^{-1}\Bigg) ,
\end{multline*}
with   $q_{(2,2)}:=q_{00}q_{10}q_{01}q_{11}$.
\end{conj}

Using an implementation of the orbifold vertex formalism into a Mathematica program, we verified these conjectures up to the following orders: 
\begin{prop}\label{prop: r=1}
Conjecture \ref{conj: closed formula orbifold vertex} holds for:
\begin{itemize}
\item $r=1$ (theorem in \cite{KR}), 
\item $r=2$ modulo $q_0^{i_0}q_1^{i_1}$ with $i_0+i_1= 6$,
\item $r=3$ modulo $q_0^{i_0}q_1^{i_1}q_2^{i_2}$ with $i_0+i_1+i_2= 5$.     
\end{itemize}
Conjecture \ref{conj: closed formula orbifold vertex D4} holds modulo $q_{00}^{i_{00}}q_{10}^{i_{10}}q_{01}^{i_{01}}q_{11}^{i_{11}}$ with $i_{00}+i_{10}+i_{01}+i_{11}=5$.
\end{prop}
Finding closed expression for the partition functions for other abelian groups $G$ appears to be an intrinsically difficult problem, already at the level of coloured box counting (cf. \cite[\S 8]{Young},\cite[\S 3.4]{DOS}, \cite[\S 5.5]{Cir}).

\subsection{A crepant resolution conjecture}

Motivated by the crepant resolution conjecture in three dimensions \cite[Conj.~A.6]{Young} and the DT/PT correspondence on Calabi-Yau 4-folds \cite{CKM}, we conjecture a \emph{crepant resolution conjecture} for Nekrasov genera. 

Let $G < \mathrm{SU}(4)$ be a finite abelian subgroup with elements of age no greater than 1. By Proposition \ref{classify g=1 abelian} and Remark \ref{rmk: noZrinSO(3)}, the only cases to consider are $\mathbb{Z}_r < \mathrm{SU}(2) < \mathrm{SU}(4)$ and $\mathbb{Z}_2 \times \mathbb{Z}_2 < \mathrm{SO}(3) < \mathrm{SU}(4)$. 
Recall the partition functions \eqref{pt partition function} and \eqref{dt partition function}, \eqref{K-orbi DT}. 
\begin{conj}\label{conj: crepant resolution for CY4}
Let $G=\BZ_r<\mathrm{SU}(2) < \mathrm{SU(4)}$ or $\BZ_2\times \BZ_2<\mathrm{SO}(3) < \mathrm{SU(4)}$.
Denote by $X \to \BC^4/G$ the crepant resolution given by the Nakamura $G$-Hilbert scheme and let $\xX = [\BC^4 / G]$.
There exist orientations such that
\begin{align*}
    \mathcal{Z}^{\DT}_{\xX,\oO_{\xX}}(y,q_0,\ldots,q_{|G|-1})=\mathcal{Z}^{\DT}_{X,\oO_X}(y,0,q)\cdot \mathcal{Z}^{\PT}_{X,\oO_X}(y,Q,q)\cdot \mathcal{Z}^{\PT}_{X,\oO_X}(y,Q^{-1},q),
\end{align*}
under the change of variables $Q^{\beta_i}=q_i$ for $i=1,\dots, |G|-1$, $q=q_0\cdots q_{|G|-1}$, where $\beta_i$ is the curve class corresponding to the irreducible $G$-representation $R_i$ by Reid's generalized McKay correspondence \eqref{equ on mckay}.
\end{conj}
Note that for the change of variables in  Conjecture \ref{conj: crepant resolution for CY4} to make sense, one needs to assume that the partition function $\mathcal{Z}^{\PT}_{X,\oO_X}(y,Q,q)$ of stable pair invariants 
 is the plethystic exponential of a rational function in the $Q$ variables. 
This is implied by Conjecture \ref{conj: PT of G non-abelian} and Lemma \ref{lemma: fixed locus stable pairs}.

\begin{thm}\label{thm on crc}
Assume Conjecture \ref{conj: PT of G non-abelian} holds for $\BZ_r$ (resp.~for $\BZ_2\times \BZ_2$), and Conjecture \ref{conj: closed formula orbifold vertex} (resp.~Conjecture~\ref{conj: closed formula orbifold vertex D4}) holds. Then Conjecture \ref{conj: crepant resolution for CY4} holds for $G=\BZ_r$ (resp.~$\BZ_2\times \BZ_2$).
\end{thm}
\begin{proof}
The proof consists of an explicit comparison of both sides under the given change of variables. This uses the fact that $\mathcal{Z}^{\DT}_{X,\mathcal{O}_X}(y,0,q)$ is determined by the localization formula, which reduces it to the toric charts, where it is given by Proposition \ref{prop: r=1} for $r=1$.
\end{proof}
\begin{rmk}\label{rmk: extending CRC global}
Conjecture \ref{conj: crepant resolution for CY4} is formulated for global quotient stacks $[\BC^4/\BZ_r], [\BC^4/\Z_2\times \BZ_2]$. One can generalize this conjecture to include more general Calabi-Yau 4-orbifolds $\xX$. The key ingredients for its formulation are:
\begin{enumerate}
    \item The existence of a crepant resolution $\widetilde{X}\to X$ of the coarse moduli space $\xX\to X$, such that it contracts at most proper curves.
    \item A suitable identification of variables of the generating series induced by a natural Fourier-Mukai transform $D^{\mathrm{b}}(\widetilde{X})\to D^{\mathrm{b}}(\xX)$ (cf.~\cite{IN1, BKR, CT, BCR}).
     \item In the non-projective case, a torus action (preserving the Calabi-Yau volume form) with proper fixed locus.
\end{enumerate}
We plan to come back to a precise generalization of Conjecture \ref{conj: crepant resolution for CY4} in a future work.
\end{rmk}

\section{Specializations of the invariants}\label{sec: Limit of theory}

\subsection{Dimensional reduction}\label{sec: dim red}

Let $X=Y\times \BC$, where $Y$ is  a Calabi-Yau 3-fold. Then $M=P_n(Y, \beta)$, $\Hilb^n(Y,\beta)$ carry symmetric perfect obstruction theories in the sense of \cite{BF, LT, Beh}. Therefore, we have a twisted virtual structure sheaf
\begin{align*}
    \widehat{\oO}^{\vir}=\oO^{\vir}\otimes K^{1/2}_{\vir}\in K_0\left(M, \BZ\left[\tfrac{1}{2}\right] \right),
\end{align*}
where $K_\vir=\det(T_M^{\vir})^{-1}$ is the virtual canonical bundle and $K_\vir^{1/2}$ is a square root \cite{NO}. Similarly, let $\xX=\yY\times \BC$, where $\yY$ is a Calabi-Yau 3-orbifold. Then $\Hilb^\alpha(\yY)$ has a symmetric perfect obstruction theory and twisted virtual structure sheaf $ \widehat{\oO}^{\vir}$ for any $\alpha\in F_0K(\yY)$.

We define the generating series of \emph{Nekrasov-Okounkov $K$-theoretic invariants} \cite{NO}:
\begin{align*}
\mathcal{Z}^{\PT}_Y(Q,q)&:=\sum_{\beta,n}\chi\left( P_{n}(Y, \beta),\widehat{\oO}^{\vir}\right)Q^{\beta}q^n, \\
\mathcal{Z}^{\DT}_Y(Q,q)&:=\sum_{\beta,n}\chi\left( \Hilb^n(Y,\beta),\widehat{\oO}^{\vir}\right)Q^{\beta}q^n, \\
\mathcal{Z}^{\DT}_{\yY}(q)&:=\sum_{\alpha\in F_0K(\yY)}\chi\left( \Hilb^\alpha(\yY),\widehat{\oO}^{\vir}\right)q^\alpha, 
\end{align*}
where $Q$, $q$ are a multi-index variables. These invariants are defined as long as the moduli space is proper, or endowed with a torus action with proper fixed locus. In the latter case, the invariants are defined equivariantly by virtual localization \cite{GP,FG} as in Definition \ref{Nekgen}. 

In \cite[App.~B]{CKM}, we conjectured a formula for the Nekrasov genera of the local resolved conifold $\oO_{\mathbb{P}^1}(-1,-1,0)$. Setting $L=\oO_X, y=t_4$, the formula in loc.~cit.~reduces to the generating series of Nekrasov-Okounkov invariants of the resolved conifold $\oO_{\mathbb{P}^1}(-1,-1)$. This was proved by studying the behaviour of the 4-fold vertex and edge terms under the specialization $L=\oO_X,  y=t_4$ \cite[\S 2.1]{CKM}. Similar techniques yield the following \emph{dimensional reduction}.\footnote{Let $X = Y \times \mathbb{C}$, where $Y$ is a toric Calabi-Yau 3-fold and consider a $(\mathbb{C}^*)^4$-fixed stable pair $(F,s)$ on $X$. The contribution of $(F,s)$ to the invariant is calculated using vertex/edge terms \cite{CKM}. In general, in a toric chart $U \cong \mathbb{C}^4$ , the specialization $y = t_4$ creates a pole in the vertex term corresponding to $U$ (unless $U$ contains only one leg like in the local resolved conifold case). However, if the underlying reduced curve of $(F,s)$ is a linear chain of $\mathbb{P}^1$'s (so every toric chart contains at most two legs), then the \emph{combined} vertex and edge terms of all charts contain no poles and the specialization $y=t_4$ is well-defined. If additionally $(F,s)$ is \emph{not} scheme theoretically supported on $Y \times \{0\} \subset Y \times \mathbb{C}$, then it contributes zero under the specialization $y=t_4$.}

\begin{cor} \label{cor:dimred1}
Assume Conjecture \ref{conj: PT of G non-abelian} holds for $G=\BZ_{r}<\mathrm{SU}(2) < \mathrm{SU}(3)$ and let $Y\to \BC^3/\BZ_{r}$ be the crepant resolution given by the Nakamura $G$-Hilbert scheme. Assume\footnote{See also \cite[Rem.~1.18]{CKM}.} additionally that the signs of the fixed points $(F,s)$ scheme theoretically supported on $Y \times \{0\} \subset Y \times \BC$ is $(-1)^{\chi(F)}$. Then
\[\mathcal{Z}^{\PT}_{Y}(Q,-q)= \mathrm{Exp}\Bigg(\sum_{0 < i\leq j < r}\frac{[t_1t_2]}{[t_3]}\cdot\frac{ Q_i\cdots Q_j}{[\kappa^{\frac{1}{2}} q] [\kappa^{\frac{1}{2}} q^{-1}]} \Bigg),\]
where $\kappa=t_1t_2t_3$ is the Calabi-Yau weight.
\end{cor}
We expect that this formula can be proved using the techniques developed by Kononov-Okounkov-Osinenko \cite{KOO} for the resolved conifold case.

Let $G < \mathrm{SU}(4)$ be a finite abelian subgroup with elements of age no greater than 1. As noted before, by Proposition \ref{classify g=1 abelian} and Remark \ref{rmk: noZrinSO(3)}, the only cases to consider are $\mathbb{Z}_r < \mathrm{SU}(2) < \mathrm{SU}(4)$ and $\mathbb{Z}_2 \times \mathbb{Z}_2 < \mathrm{SO}(3) < \mathrm{SU}(4)$.
    Let $Y\to \BC^3/G$ be the Nakamura $G$-Hilbert scheme. Then, specializing $y=t_4$, we expect a similar dimensional reduction principle to hold for  $X=Y\times \BC$.
    In this more general setting, one cannot exploit the vertex formalism, but one may be able to apply a version of the Lefschetz principle developed in \cite[Cor.~B.4]{Park}.

On the orbifold side, the dimensional reduction for the ordinary vertex term \cite[Prop.~2.1]{CKM} immediately adapts to orbifold vertex (and in fact works for any Calabi-Yau 4-orbifold $[\BC^4/G]\cong [\BC^3/G]\times \BC $, where $G$ is \emph{any} finite abelian subgroup of $\mathrm{SU}(3)$). Setting $y=t_4$ in Conjectures \ref{conj: closed formula orbifold vertex}, \ref{conj: closed formula orbifold vertex D4}, we recover the conjectural expressions of the orbifold 3-fold vertex of Cirafici \cite[Eqn.~(5.37),~(5.56)]{Cir} in string theory.
\begin{cor} \label{cor:dimred2}
Assume Conjecture \ref{conj: closed formula orbifold vertex} holds.  Assume additionally that the signs of the fixed points $W \subset \BC^4$ scheme theoretically supported on $\BC^3 \times \{0\} \subset \BC^4$ is $(-1)^{\dim H^0(W,\mathcal{O}_W)^G}$. 
Then 
\begin{align*}
\mathcal{Z}^{\DT}_{[\mathbb{C}^3/\BZ_{r}]}(-q_0,q_1,\ldots,q_{r-1}) = \Exp\left(\mathcal{F}^{\mathrm{3D}}_r(t, q)+\mathcal{F}^{\mathrm{col},\mathrm{3D}}_{r}(t,q_0,\dots, q_{r-1})\right),
\end{align*}
where 
\begin{align*}
\mathcal{F}^{\mathrm{3D}}(t, q)&:=\frac{[t_1t_2][t_2t_3][t_1t_3]}{[t_1][t_2][t_3]}\cdot\frac{1}{[\kappa^{\frac{1}{2}}q][\kappa^{\frac{1}{2}}q^{-1}]},\\
    \mathcal{F}^{\mathrm{3D}}_r(t_1,t_2,t_3, q)&:=\sum_{k=0}^{r-1}\mathcal{F}^{\mathrm{3D}}\left(t_{1}^{r-k}t_2^{-k}, t_2^{k+1}t_1^{-r+k+1}, t_3,  q_{(r)}\right),\\
    \mathcal{F}^{\mathrm{col}, \mathrm{3D}}_{r}(t,q_0,\dots, q_{r-1})&:=\sum_{0<i\leqslant  j<r}\left(q_{[i,j]}+q_{[i,j]}^{-1}\right)\frac{[t_1 t_2]}{[t_3][\kappa^{\frac{1}{2}}q_{(r)}][\kappa^{\frac{1}{2}}q_{(r)}^{-1}]},
\end{align*}
with $\kappa:=t_1t_2t_3$. 

Assume Conjecture \ref{conj: closed formula orbifold vertex D4} holds. Then 
\begin{align*}
\mathcal{Z}^{\DT}_{[\mathbb{C}^3/\BZ_{2}\times \BZ_2]}(-q_{00},q_{10}, q_{01}, q_{11}) = \Exp\left(\mathcal{F}^{\mathrm{3D}}_{2,2}(t, q_{(2,2)})+\mathcal{F}^{\mathrm{col,\mathrm{3D}}}_{2,2}(t,q_{00},q_{10},q_{01}, q_{11}) \right),
\end{align*}
where 
\begin{multline*}
     \mathcal{F}^{\mathrm{3D}}_{2,2}(t_1,t_2,t_3, q):=\mathcal{F}^{\mathrm{3D}}\left(t_1^2, t_2^{2}, \frac{t_3}{t_1t_2}, q\right)+\mathcal{F}^{\mathrm{3D}}\left(t_1^2, \frac{t_2}{t_1t_3}, t_3^{2}, q\right)\\
    +\mathcal{F}^{\mathrm{3D}}\left(\frac{t_1}{ t_2t_3}, t_2^{2}, t_3^{2}, q\right)+\mathcal{F}^{\mathrm{3D}}\left(\frac{t_2t_3}{t_1}, \frac{t_1t_3}{t_2},\frac{t_1t_2}{t_3}, q\right),
\end{multline*}
\begin{multline*}
     \mathcal{F}^{\mathrm{col,\mathrm{3D}}}_{2,2}(t,q_{00},q_{10},q_{01}, q_{11}):=\frac{1}{[\kappa^{\frac{1}{2}}q_{(2,2)}][\kappa^{\frac{1}{2}}q_{(2,2)}^{-1}]}\Bigg(\frac{[t_1t_2t_3^{-1}]}{[t_3^{2}]}(q_{10}q_{01}+q_{10}^{-1}q_{01}^{-1})\\
     +\frac{[t_1t_2^{-1}t_3]}{[t_2^{2}]}\left(q_{10}q_{11}+q_{10}^{-1}q_{11}^{-1}\right)+\frac{[t_1^{-1}t_2t_3]}{[t_1^{2}]}\left(q_{01}q_{11}+q_{01}^{-1}q_{11}^{-1}\right)\\
     +q_{10}+q_{01}+q_{11}+q_{10}q_{01}q_{11}+q_{10}^{-1}+q_{01}^{-1}+q_{11}^{-1}+q_{10}^{-1}q_{01}^{-1}q_{11}^{-1}\Bigg) ,
\end{multline*}
with $\kappa:=t_1t_2t_3$.
\end{cor}

Combining Corollaries \ref{cor:dimred1} and \ref{cor:dimred2}, we obtain a \emph{crepant resolution correspondence} for Nekrasov-Okounkov $K$-theoretic invariants for 
$G=\BZ_r<\mathrm{SU}(2)$, which ---to our knowledge--- is new. 
\begin{cor}
Assume Conjecture \ref{conj: crepant resolution for CY4} holds for $G=\BZ_{r}<\mathrm{SU}(2) < \mathrm{SU}(3)$. Denote by $Y\to \BC^3/\BZ_{r}$ the crepant resolution given by the Nakamura $G$-Hilbert scheme and let $\mathcal{Y} = [\BC^3 / \BZ_r]$. Assume,  additionally, that the signs of the fixed points $(F,s)$ scheme-theoretically supported on $Y \times \{0\} \subset Y \times \BC$ is $(-1)^{\chi(F)}$ and the signs of the fixed points $W \subset \BC^4$ scheme theoretically supported on $\BC^3 \times \{0\} \subset \BC^4$ is $(-1)^{\dim H^0(W,\mathcal{O}_W)^G}$.
Then 
\begin{align*}
    \mathcal{Z}^{\DT}_{\mathcal{Y}}(q_0,\ldots,q_{r-1})=\mathcal{Z}^{\DT}_{Y}(0,q)\cdot \mathcal{Z}^{\PT}_{Y}(Q,q)\cdot \mathcal{Z}^{\PT}_{Y}(Q^{-1},q),
\end{align*}
under the change of variables $Q^{\beta_i}=q_i$ for $i=1,\dots, |G|-1$ and $q=q_0\cdots q_{r-1}$. Here $\beta_i$ is the curve class corresponding to the irreducible $\BZ_r$-representation $R_i$ under the McKay correspondence \eqref{equ on mckay}.
\end{cor}

\subsection{Cohomological limit}\label{sec: cohom limit}

Let $X$  be a  Calabi-Yau 4-fold, endowed with a trivial $\BC^*$-action with irreducible character $y$ and $m:=c_1^{\C^*}(y)$.
Suppose $X$ is projective, or more generally $X$ is quasi-projective with action by an algebraic torus $\TT$ preserving the Calabi-Yau volume form, and such that the $\TT$-fixed loci of the $\DT$ moduli spaces are proper (as in \S \ref{sec: DT4 virtual struc}). 
Let $L$ be a $\TT$-equivariant line bundle on $X$.
We define $\DT$ invariants in $\TT$-equivariant cohomology by\footnote{Here the Euler class is defined as follows. In the $K$-group of locally free sheaves, write $(L^{[n]})^\vee = E_0 - E_1$ for locally free sheaves $E_0, E_1$. Then $e((L^{[n]})^\vee\otimes y) = e(E_0 \otimes y) / e(E_1 \otimes y)$, where on the right-hand side the Euler classes are equivariant with respect to the trivial $\mathbb{C}^*$-action. When the invariants are defined by virtual localization on the $\TT$-fixed locus, the Euler classes on the right-hand side are $\TT \times \mathbb{C}^*$-equivariant.}
\begin{align}\label{eqn: inv for cohom}
\begin{split} 
I_{n,\beta}^{\mathrm{coh}}(X,L,m)&:=\int_{[\Hilb^n(X, \beta)]^{\vir}}e((L^{[n]})^\vee\otimes y)\in H^*_{\TT\times \BC^*}(\pt)_{\mathrm{loc}},\\
\mathcal{Z}^{\DT, \mathrm{coh}}_{X,L}(m,Q,q)&:=\sum_{\beta,n} I^{\mathrm{coh}}_{n,\beta}(X,L, m)\,Q^{\beta}q^n\in  H^*_{\TT\times \BC^*}(\pt)_{\mathrm{loc}}(\!( q,Q)\!),
\end{split}
\end{align}
and similarly for PT invariants $P_{n,\beta}^{\mathrm{coh}}(X,L,m)$ and its generating series $ \mathcal{Z}^{\PT, \mathrm{coh}}_{X,L}(m,Q,q)$.
These invariants were studied in \cite{CK1, CT3}, \cite[\S 0.4]{CKM}, \cite[\S 2.3]{CT4}. 

For example, if $X$ is a toric Calabi-Yau 4-fold acted on by the subtorus $\TT < (\C^*)^4$ preserving the Calabi-Yau volume form, we have
\[H^*_{\TT\times \BC^*}(\pt)_{\mathrm{loc}}=\frac{\BQ(\lambda_1, \lambda_2, \lambda_3, \lambda_4,m)}{(\lambda_1+\lambda_2+\lambda_3+\lambda_4)},\]
where we set $\lambda_i=c_1^{\TT}(t_i)$ for $i=1,\dots, 4$ and the invariants are defined by the (cohomological) virtual localization formula on the $\TT$-fixed locus of Oh-Thomas \cite{OT}. Similarly we define orbifold DT invariants $I^{\mathrm{coh}}_{\alpha}(\xX,L,m)$ for a Calabi-Yau 4-orbifold $\xX$ and class $\alpha\in F_0K(\xX)$, and its generating series 
$\mathcal{Z}^{\DT, \mathrm{coh}}_{\xX,L}(m,q)$.

In \cite[\S 2.2]{CKM}, we proved that for a toric Calabi-Yau 4-fold $X$ and suitable orientations, a certain limit of the Nekrasov genus computes the invariants in equivariant cohomology as follows
\begin{align*}
    \lim_{b\to 0}\left. I_{n,\beta}(X,L,y)\right|_{t_i=e^{b\lambda_i}, y=e^{bm}}=I_{n,\beta}^{\mathrm{coh}}(X,L,m),
\end{align*}
where $(\cdot)|_{t_i=e^{b\lambda_i},y=e^{bm}}$ means evaluation of $(\cdot)$ in $t_i = e^{b \lambda_i}$ and $y=e^{bm}$, and we first evaluate and then take the limit. See \cite[Prop. 2.2]{Boj2} for a similar result in the compact setting. An analogous statement was derived for PT invariants in \cite{CKM}, where we assume that the $(\BC^*)^4$-fixed loci of the PT moduli spaces (for fixed $X,\beta$ and all $n$) is 0-dimensional. This assumption ensures that the $(\BC^*)^4$-fixed loci coincide with the $\TT$-fixed loci and are 0-dimensional reduced \cite[Prop.~2.6]{CK2}. 
The following theorem simultaneously generalizes the toric and projective cases.
\begin{thm}\label{thm: cohom limit}
Let $X$ be a Calabi-Yau 4-fold. Let $\TT$ be a (possibly trivial) algebraic torus acting on $X$ preserving the Calabi-Yau volume form. Let $L$ be a $\TT$-equivariant line bundle on $X$. If $\Hilb^n(X, \beta)^\TT$ is proper, then we have 
    \begin{align*}
    \lim_{b\to 0}\left. I_{n,\beta}(X,L,y)\right|_{t_i=e^{b\lambda_i}, y=e^{bm}}&=I_{n,\beta}^{\mathrm{coh}}(X,L,m).
\end{align*}
If $P_n(X, \beta)^\TT$ is proper, then we have 
\begin{align*}
     \lim_{b\to 0}\left. P_{n,\beta}(X,L,y)\right|_{t_i=e^{b\lambda_i}, y=e^{bm}}&=P_{n,\beta}^{\mathrm{coh}}(X,L,m).
\end{align*}
\end{thm}
\begin{proof}
Let $I=\Hilb^n(X, \beta)$ and decompose its fixed locus into connected components $I^\TT=\bigsqcup_\alpha I_\alpha$. Define $\fF:=\det(L^{[n]} \otimes y^{-1})^{-1/2}$. By  the (equivariant) virtual Riemann-Roch theorem and virtual localization formula of Oh-Thomas \cite[Thm.~6.1, Thm.~7.1]{OT}, we have 
   \begin{align*}
       \chi\left(I, \widehat{\oO}_I^{\vir}\otimes \widehat{\Lambda}^{\mdot}  (L^{[n]} \otimes y^{-1})\right) 
       &=\int_{[I]^{\vir}} \ch( \Lambda^{\mdot} (L^{[n]} \otimes y^{-1})) \cdot \sqrt{\td}(T_I^{\vir}) \cdot \ch\left(\fF\right) \\
       &=\sum_{\alpha}\int_{[I_\alpha]^{\vir}}\frac{\ch( \Lambda^{\mdot} (L^{[n]}|_{I_\alpha} \otimes y^{-1}))}{\sqrt{e}(N^{\vir})}\cdot \sqrt{\td}(T_I^{\vir}|_{I_\alpha}) \cdot \ch\left(\fF|_{I_\alpha}\right) \\
       &=\sum_{\alpha}\int_{[I_\alpha]^{\vir}}\frac{e((L^{[n]}|_{I_\alpha})^\vee \otimes y)}{\sqrt{e}(N^{\vir})}\cdot\frac{\sqrt{\td}(T_I^{\vir}|_{I_\alpha})}{\td(L^{[n]}|_{I_\alpha}^\vee \otimes y)})\cdot \ch\left(\fF|_{I_\alpha}\right),
   \end{align*}
   where the integration map  is defined equivariantly, $N^{\vir}$ is the virtual normal bundle of $I_\alpha\subset I$, and $T_I^{\vir}$ is the virtual tangent bundle of $I$.
   All Chern characters, Euler classes, and Todd classes in these formulae are equivariant with respect to the algebraic torus $\widetilde{\TT}= \TT\times \BC^*$ with character group  $\widehat{\TT}\oplus \BZ$, where $\BC^*$ acts trivially on $I$ with primitive character $y$. 
   For the third equality, we used the identity $\ch(\Lambda^\mdot V) = e(V^*) / \td(V^*)$ for any $\widetilde{\TT}$-equivariant vector bundle $V$.
   We use multi-index notation $t^\mu \in \widehat{\TT}\oplus \BZ$ where $t^\mu = y^{\mu_0} \prod_i t_i^{\mu_i}$.
   In $K^0_{\widetilde{\TT}}(I_\alpha)$, we have weight decompositions
   \begin{align*}
       (L^{[n]}|_{I_\alpha})^\vee \otimes y&=\bigoplus_{\mu}L_\mu\otimes t^{\mu},\\
       T_I^{\vir}|_{I_\alpha}&=T_{I_\alpha}^{\vir}\oplus N^{\vir},\\
       N^{\vir}&=\bigoplus_{\mu}(N_\mu\otimes t^{\mu})\oplus \bigoplus_{\mu}(N_\mu^\vee\otimes t^{-\mu}),
   \end{align*}
for certain non-zero classes $L_\mu, N_\mu \in K^0(I_\alpha)$.
In the last line we used that $N^{\vir}$ is $\widetilde{\TT}$-moving and self-dual, so all its weight spaces come in pairs (see also \cite[\S 7]{OT}).    
   In particular, the virtual tangent bundle $T_{I_\alpha}^{\mathrm{vir}}$ of $I_\alpha$ is $\widetilde{\TT}$-fixed, and $(L^{[n]}|_{I_\alpha})^\vee \otimes y$ and $N^{\vir}$ are $\widetilde{\TT}$-moving.
   Then
   \begin{align*}
   \sqrt{e}(N^{\vir})&= e\left(\bigoplus_{\mu} N_\mu\otimes t^{\mu}\right),\\
   \sqrt{\td}(T_I^{\vir}|_{I_\alpha})&= \td\left(\bigoplus_{\mu}N_\mu\otimes t^{\mu}\right) \cdot \sqrt{\td}(T_{I_\alpha}^{\vir}).
   \end{align*}
 Define
   \begin{align*}
       W:=\bigoplus_{\mu}L_\mu\otimes t^{\mu}-\bigoplus_{\mu} N_\mu\otimes t^{\mu}
   \end{align*}
  and notice that $W$ is $\widetilde{\TT}$-movable. It follows that
   \begin{align*}    \int_{[I_\alpha]^{\vir}}\frac{e((L^{[n]}|_{I_\alpha})^\vee \otimes y)}{\sqrt{e}(N^{\vir})}\cdot\frac{\sqrt{\td}(T_I^{\vir}|_{I_\alpha})}{\td(L^{[n]}|_{I_\alpha}^\vee \otimes y)}\cdot \ch\left(\fF|_{I_\alpha}\right)=\int_{[I_\alpha]^{\vir}}e(W)\cdot \td(-W)\cdot \sqrt{\td}(T_{I_\alpha}^{\vir})\cdot \ch\left(\fF|_{I_\alpha}\right).
   \end{align*}
   Let $m:=c_1^{\widetilde{\TT}}(y)$ and $\lambda_i = c_1^{\widetilde{\TT}}(t_i)$. 
   Denote by $n_\alpha=\mathrm{vdim} (I_\alpha) = \tfrac{1}{2} \rk(T_{I_\alpha}^{\vir})$ the virtual dimension of $I_\alpha$. Since $\rk L^{[n]}=n$ and $\rk N^{\vir}=2n-2n_\alpha $, it is easy to see that 
   \begin{align*}
      e(W)=\sum_{j\geq 0} \sum_{k=0}^{l_j} \beta_{j,k}\otimes \gamma_{j,k} \in  A^*(I_\alpha)\otimes \BQ(m,\{\lambda_i\}_i),
   \end{align*}
   where $\beta_{j,k}\in A^j(I_\alpha)$ and $\gamma_{j,k} \in \BQ(m,\{\lambda_i\}_i)$ is a \emph{homogeneous} rational function of degree $n_\alpha-j$ for all $k$. 
   %In this proof, we crucially use that $\rk(L^{[n]}) = \mathrm{vd}(I) = n$.
   Moreover we have
   \begin{align*}
       \td(-W)\cdot \sqrt{\td}(T_{I_\alpha}^{\vir})\cdot \ch\left(\fF|_{I_\alpha}\right)=1+\sum_{j\geq 0}\delta_j\otimes \epsilon_j\in A^*(I_\alpha)\otimes \BQ\llbracket m, \{\lambda_i\}_i\rrbracket,
   \end{align*}
   where $\delta_j \in A^j(I_\alpha)$ and $\epsilon_j \in \BQ\llbracket m, \{\lambda_i\}_i\rrbracket$.
Furthermore the term $\delta_0 \otimes \epsilon_0 \in \BQ\llbracket m, \{\lambda_i\}_i\rrbracket$ has no constant term. 
   Substituting $m \mapsto b m$ and $\lambda_i \mapsto b \lambda_i$, we obtain
   \begin{align*}
       \int_{[I_\alpha]^{\vir}}e(W)\cdot \td(-W)\cdot \sqrt{\td}(T_{I_\alpha}^{\vir})\cdot \ch(\fF|_{I_\alpha}) = \sum_{k=0}^{l_{n_\alpha}} \gamma_{n_\alpha,k}    \cdot\int_{[I_\alpha]^{\vir}}\beta_{n_\alpha,k}+ O(b),
   \end{align*}
   where $O(b)$ is an element of $b \cdot \BQ\llbracket b, m, \{\lambda_i\}_i\rrbracket$, which goes to zero as $b \mapsto 0$. Hence the desired limit yields
   \begin{align*}
      \int_{[I_\alpha]^{\vir}}e(W)=\int_{[I_\alpha]^{\vir}}\frac{e((L^{[n]}|_{I_\alpha})^\vee \otimes y)}{\sqrt{e}(N^{\vir})},
   \end{align*}
which concludes the proof by summing over all the connected components and using the Oh-Thomas localization formula once more. 
The proof for the PT case proceeds analogously.
\end{proof}
Denote by
\begin{align*}
    M(Q,q):=\prod_{n=1}^\infty (1-Qq^n)^{-n}= \Exp\left(\frac{Qq}{(1-q)^2} \right) 
\end{align*}
the refined MacMahon series. By Theorem \ref{thm: cohom limit}, we can express cohomological DT/PT invariants  in terms of $M(Q,q)$ (see also \cite[App.~A]{CKM} and \cite[Thm.~7.2]{FMR}).
\begin{cor}
Assume Conjecture \ref{conj: PT of G non-abelian} holds. Then there exist orientations such that
\[
\mathcal{Z}^{\PT,\mathrm{coh}}_{X, \oO_X}(m,Q,q)= \prod_{\beta\in H_2(X,\mathbb{Z})}M(Q^{\beta},q)^{P^{\mathrm{coh}}_{1,\beta}(X,\oO_X,\lambda_4)\cdot\frac{m}{\lambda_4}}.
\]
\end{cor}
All the stable pair invariants $P^{\mathrm{coh}}_{X,1,\beta}(X,\oO_X,\lambda_4)$ are determined by limits of their $K$-theoretic counterparts studied in \S \ref{sec: PT invariants resolution}.

The cohomological limit also applies to the $K$-theoretic DT invariants of the orbifolds $[\BC^4/\BZ_r]$, $[\BC^4/\BZ_2\times \BZ_2]$ studied in Conjectures \ref{conj: closed formula orbifold vertex} and \ref{conj: closed formula orbifold vertex D4}.\footnote{Here it is important that for $\xX = [\BC^4 / G]$ and $G = \BZ_r$, $\BZ_2 \times \BZ_2$, we have $\rk (L^{[R]}) = \mathrm{vd}_{\mathbb{C}} (\Hilb^R(\xX)) = d_0$ for any finite-dimensional representation $R$ of $G$ containing $d_0$ copies of the trivial representation $R_0$ (\S \ref{sec: orbivertexexamples}). This does not hold for all finite abelian groups $G < \mathrm{SU}(4)$.} Consequently, we can derive expressions for the partition functions of their cohomological DT invariants. We set
\[
\widetilde{M}(Q,q):=M(Q,q)\cdot M(Q^{-1},q).
\]
\begin{cor}\label{cor: cohom D4}
Assume Conjecture \ref{conj: closed formula orbifold vertex} holds. Then there exist orientations such that 
\begin{multline*}
\mathcal{Z}^{\DT, \mathrm{coh}}_{[\mathbb{C}^4/\BZ_{r}],\oO}(m,q_0,\ldots,q_{r-1}) =M(1,q_{(r)})^{-\frac{m}{\lambda_4}\left(\frac{r(\lambda_1+\lambda_2)}{\lambda_3}+\frac{(\lambda_1+\lambda_2)(\lambda_1+\lambda_2+\lambda_3)}{r \lambda_1\lambda_2} \right)}\\
\cdot \prod_{0<i\leq j<r}\widetilde{M}(q_{[i,j]},q_{(r)}) ^{-\frac{m}{\lambda_4}\cdot \frac{(\lambda_1+\lambda_2)}{\lambda_3}}.
\end{multline*}
Assume Conjecture \ref{conj: closed formula orbifold vertex D4} holds. Then there exist orientations such that 
\begin{multline*}
\mathcal{Z}^{\DT,\mathrm{coh}}_{[\mathbb{C}^4/\BZ_{2}\times \BZ_2],\oO}(m,q_{00},q_{10}, q_{01}, q_{11}) =\Big(M(1,q_{(2,2)})^{-1+\frac{\lambda_1}{\lambda_2}+\frac{\lambda_2}{\lambda_1}+\frac{\lambda_1}{\lambda_3}+\frac{\lambda_3}{\lambda_1}+\frac{\lambda_2}{\lambda_3}+\frac{\lambda_3}{\lambda_2}} \\  \cdot\widetilde{M}(q_{10},q_{(2,2)})\cdot \widetilde{M}(  q_{01},q_{(2,2)})\cdot \widetilde{M}(q_{11},q_{(2,2)})\cdot \widetilde{M}(q_{10}q_{01}q_{11},q_{(2,2)})\\
\cdot  \widetilde{M}(q_{01}q_{10},q_{(2,2)})^{\frac{\lambda_1+\lambda_2-\lambda_3}{2\lambda_3}}\cdot  \widetilde{M}(q_{10}q_{11},q_{(2,2)})^{\frac{\lambda_1-\lambda_2+\lambda_3}{2\lambda_2}}\cdot  \widetilde{M}(q_{01}q_{11},q_{(2,2)})^{\frac{-\lambda_1+\lambda_2+\lambda_3}{2\lambda_1}}\Big)^{-\frac{m}{\lambda_4}}.
\end{multline*}
\end{cor}
 The formulae in Corollary  \ref{cor: cohom D4} were independently found in \cite[Prop.~4.7, Conj.~5.10]{ST2} from the viewpoint of string theory.  
 
 Applying the cohomological limit, we obtain a \emph{crepant resolution correspondence} for cohomological invariants with tautological insertions. 
\begin{cor}
Assume Conjecture \ref{conj: crepant resolution for CY4} holds. Then there exist orientations such that 
\begin{align*}
    \mathcal{Z}^{\DT,\mathrm{coh}}_{\xX,\oO_{\xX}}(m,q_0,\ldots,q_{|G|-1})=\mathcal{Z}^{\DT,\mathrm{coh}}_{X,\oO_X}(m,0,q)\cdot \mathcal{Z}^{\PT,\mathrm{coh}}_{X,\oO_X}(m,Q,q)\cdot \mathcal{Z}^{\PT,\mathrm{coh}}_{X,\oO_X}(m,Q^{-1},q),
\end{align*}
under the change of variables $Q^{\beta_i}=q_i$, $q=q_0\cdots q_r$, where $\beta_i$ is the curve class corresponding to the irreducible $G$-representation $R_i$ by Reid's generalized McKay correspondence \eqref{equ on mckay}.
\end{cor}
Let $X = Y \times \BC$ and $\xX=\yY\times \BC$, where $Y$ is a toric Calabi-Yau 3-fold and $\yY$ is a toric Calabi-Yau 3-orbifold. 
Setting $m=\lambda_4$, the results of this subsection dimensionally reduce to cohomological DT/PT invariants of $Y$, $\mathcal{Y}$ respectively.
For $\mathcal{Y} = [\BC^3/\BZ_r]$, the resulting formula coincides with an expression of Z.~Zhou \cite[Thm.~5.3]{Zhou}\footnote{Beware of a small typo in \cite[Thm.~5.3]{Zhou}.}, which generalizes \cite[Thm.~1]{MNOP2}. 
For $\mathcal{Y} = [\BC^3/\BZ_2\times \BZ_2]$, an analogous formula can be proved by adapting Zhou's strategy and using relative DT invariants. 
Further specializing the equivariant parameters to the \emph{Calabi-Yau torus} $\lambda_1+\lambda_2+\lambda_3=0$, all the equivariant parameters drop out and the resulting expressions coincide with the ordinary (numerical) DT invariants of $Y$, $\mathcal{Y}$ defined and computed using Behrend's constructible function in \cite[App.~A]{Young}, \cite[Cor.~2.8]{Cal}, \cite[Thm.~1.3]{Toda}.

\subsection{Insertion-free limit}

Let $X$  be a  Calabi-Yau 4-fold. As in \S \ref{sec: cohom limit}, suppose $X$ is projective, or $X$ is quasi-projective with action by an algebraic torus $\TT$ preserving the Calabi-Yau volume form, and such that the $\TT$-fixed loci of the $\DT$ moduli spaces are proper. 
Let $L$ be a $\TT$-equivariant line bundle on $X$.
We define \emph{insertion-free} DT invariants in ($\TT$-equivariant) cohomology
\begin{align*}
\begin{split}
    I_{n,\beta}^{\mathrm{free}}(X)&:=\int_{[\Hilb^n(X, \beta)]^{\vir}}1\in H^*_{\TT}(\pt)_{\mathrm{loc}},\\
\mathcal{Z}^{\DT, \mathrm{free}}_{X}(Q,q)&:=\sum_{\beta,n} I^{\mathrm{free}}_{n,\beta}(X)Q^{\beta}q^n\in  H^*_{\TT}(\pt)_{\mathrm{loc}}(\!(q,Q)\!).
\end{split}
\end{align*}
 Similarly we define PT invariants 
$ P_{n,\beta}^{\mathrm{free}}(X)$, orbifold DT invariants $I^{\mathrm{free}}_{\alpha}(\xX)$ for a Calabi-Yau 4-orbifold $\xX$ and class $\alpha\in F_0K(\xX)$, and their generating series $ \mathcal{Z}^{\PT, \mathrm{free}}_{X}(Q,q)$, $\mathcal{Z}^{\DT, \mathrm{free}}_{\xX}(q)$. 

The (complex) virtual dimension of $\Hilb^n(X, \beta)$, $P_n(X,\beta)$ is $n$ \cite[\S 0.5]{CKM}, \cite[\S 6]{CT1}, \cite[Prop.~2.6~(2)]{CT4}.
Hence we focus on the case where $\TT$ acts non-trivially.
We showed in \cite[Thm.~0.15]{CKM} that for a toric Calabi-Yau 4-fold $X$, we have
\[
\lim_{m\to \infty}\mathcal{Z}^{\DT, \mathrm{coh}}_{X,\oO_X}(m,Q,q)|_{\tilde{q}=qm}=\mathcal{Z}^{\DT, \mathrm{free}}_{X}(Q,\tilde{q}).
\]
An analogous statement was derived for PT invariants in \cite{CKM}, where we assume that the $(\BC^*)^4$-fixed loci of the PT moduli spaces (for fixed $X,\beta$ and all $n$) is 0-dimensional. This assumption ensures that the $(\BC^*)^4$-fixed loci coincide with the $\TT$-fixed loci and are 0-dimensional reduced \cite[Prop.~2.6]{CK2}.
 We generalize this insertion-free limit to the non-toric setting.
\begin{thm}\label{thm: insertion-free limit}
Let $X$ be a Calabi-Yau 4-fold. Let $\TT$ be a (possibly trivial) algebraic torus acting on $X$ preserving the Calabi-Yau volume form. Let $L$ be a $\TT$-equivariant line bundle on $X$. If $\Hilb^n(X, \beta)^\TT$ is proper, then we have 
    \begin{align*}
    \lim_{m\to \infty}I_{n,\beta}^{\mathrm{coh}}(X,L,m)\cdot q^n|_{\tilde{q}=qm}&=I_{n,\beta}^{\mathrm{free}}(X)\cdot \tilde{q}^n.
\end{align*}
If $P_n(X, \beta)^\TT$ is proper, then we have 
    \begin{align*}
     \lim_{m\to \infty}P_{n,\beta}^{\mathrm{coh}}(X,L,m)\cdot q^n|_{\tilde{q}=qm}&=P_{n,\beta}^{\mathrm{free}}(X)\cdot \tilde{q}^n.
\end{align*}
\end{thm}
\begin{proof}
    Let $I_n=\Hilb^n(X, \beta)$ and decompose its $\TT$-fixed locus into connected components $I_n^\TT=\bigsqcup_\alpha I_{n,\alpha}$. By definition of the invariants we have
    \begin{align*}
        I_{n,\beta}^{\mathrm{coh}}(X,L,m)=\sum_{\alpha}\int_{[I_{n,\alpha}]^{\vir}}\frac{e\left((L^{[n]})^\vee \otimes y\right)}{\sqrt{e}(N^{\vir})},
    \end{align*}
    where $N^{\vir}$ denotes the virtual normal bundle of $I_{n,\alpha}\subset I_{n}$. Recall that $m=c_1^{\mathbb{C}^*}(y)$ (with respect to the trivial $\BC^*$ action on $I_n$) and $\rk(L^{[n]})=n$. A simple identity of Chern classes yields 
        \[
e\left((L^{[n]})^\vee \otimes y\right)=m^{n}\cdot c_{m^{-1}}\left((L^{[n]})^\vee\right),
    \]
    where $c_{m^{-1}}(x) := 1+c_1(x) \cdot m^{-1} + c_2(x) \cdot m^{-2} + \cdots$ is the total  Chern class weighted by $m^{-1}$. Applying the limit we have
    \begin{align*}
 \lim_{m\to \infty}\frac{\tilde{q}^n}{m^n}\int_{[I_{n,\alpha}]^{\vir}}\frac{e\left((L^{[n]})^\vee\otimes y\right)}{\sqrt{e}(N^{\vir})}&=\lim_{m\to \infty}\tilde{q}^n\int_{[I_{n,\alpha}]^{\vir}}\frac{c_{m^{-1}}\left((L^{[n]})^\vee\right)}{\sqrt{e}(N^{\vir})}\\
&=\tilde{q}^n\int_{[I_{n,\alpha}]^{\vir}}\frac{1}{\sqrt{e}(N^{\vir})},
    \end{align*}
    %In this proof it is crucial that the power of $q^n$ coincides with $\rk(L^{[n]})$.
    which concludes the proof by summing over all the connected components $I_{n,\alpha}$. The proof for the stable pairs case is completely analogous.
\end{proof}

Using the plethystic expression of the refined MacMahon function, we can compute its insertion-free limit
\begin{align*}
    \lim_{m\to \infty}M(Q, q)^{\frac{m}{\lambda_4}}|_{\tilde{q}=mq}&= \lim_{m\to \infty}\exp\left(\sum_{n\geq 1}\frac{1}{n}\frac{\tilde{q}^{n} Q^n}{\lambda_4 m^{n-1}(1-\frac{\tilde{q}^n}{m^{n}})^{2}} \right)\\
    &=e^{\frac{\tilde{q}Q}{\lambda_4} }.
\end{align*}
By Theorem \ref{thm: insertion-free limit}, we obtain closed formulae for the partition functions of insertion-free DT and PT invariants.
\begin{cor}
Assume Conjecture \ref{conj: PT of G non-abelian} holds. Then there exist orientations such that
\[
\mathcal{Z}^{\PT,\mathrm{free}}_{X}(Q,q)= \prod_{\beta\in H_2(X,\mathbb{Z})}e^{P^{\mathrm{free}}_{1,\beta}(X)\cdot\frac{qQ}{\lambda_4}}.
\]
\end{cor}
The stable pair invariants $ P^{\mathrm{free}}_{1,\beta}(X)$ can be easily obtained as limits of their cohomological counterparts determined in \S \ref{sec: PT invariants resolution}, \ref{sec: cohom limit}.

The insertion-free limit also applies to the $K$-theoretic DT invariants of the orbifolds $[\BC^4/\BZ_r]$, $[\BC^4/\BZ_2\times \BZ_2]$ studied in Conjectures \ref{conj: closed formula orbifold vertex} and \ref{conj: closed formula orbifold vertex D4}. Hence, we can derive expressions for the partition functions of their insertion-free DT invariants.
Recall that $\rk (L^{[R]}) = d_0$ for any finite-dimensional representation $R$ containing $d_0$ copies of the trivial representation $R_0$.
By setting $\tilde{q}_0=q_0m$ in the case of $G = \BZ_r$ (and $\tilde{q}_{00}=q_{00}m$ in the case of $G = \BZ_2 \times \BZ_2$) and sending $m\to \infty$, we obtain: 
\begin{cor}\label{cor: pure theory orbifold}
Assume Conjecture \ref{conj: closed formula orbifold vertex} holds. Then there exist orientations such that 
\begin{multline*}
\mathcal{Z}^{\DT, \mathrm{free}}_{[\mathbb{C}^4/\BZ_{r}]}(q_0,\ldots,q_{r-1}) =e^{-\frac{q_{(r)}}{\lambda_4}\left(\frac{r(\lambda_1+\lambda_2)}{\lambda_3}+\frac{(\lambda_1+\lambda_2)(\lambda_1+\lambda_2+\lambda_3)}{r \lambda_1\lambda_2} \right)}\\
\cdot \prod_{0<i\leq j<r} e^{-(q_{[i,j]}+q_{[i,j]}^{-1})\frac{q_{(r)}}{\lambda_4}\cdot \frac{(\lambda_1+\lambda_2)}{\lambda_3}}.
\end{multline*}
Assume Conjecture \ref{conj: closed formula orbifold vertex D4} holds. Then there exist orientations such that 
\begin{multline*}
\mathcal{Z}^{\DT,\mathrm{free}}_{[\mathbb{C}^4/\BZ_{2}\times \BZ_2]}(q_{00},q_{10}, q_{01}, q_{11}) =e^{q_{(2,2)}(-1+\frac{\lambda_1}{\lambda_2}+\frac{\lambda_2}{\lambda_1}+\frac{\lambda_1}{\lambda_3}+\frac{\lambda_3}{\lambda_1}+\frac{\lambda_2}{\lambda_3}+\frac{\lambda_3}{\lambda_2})} \\  \cdot e^{q_{(2,2)}(q_{10}+q_{10}^{-1}+q_{01}+q_{01}^{-1}+q_{11}+q_{11}^{-1}+q_{10}q_{01}q_{11}+q_{10}^{-1}q_{01}^{-1}q_{11}^{-1})}\\ \cdot e^{q_{(2,2)}\left((q_{01}q_{10} + q_{01}^{-1}q_{10}^{-1})\frac{\lambda_1+\lambda_2-\lambda_3}{2\lambda_3}+(q_{10}q_{11}+q_{10}^{-1}q_{11}^{-1})\frac{\lambda_1-\lambda_2+\lambda_3}{2\lambda_2}+ (q_{01}q_{11}+q_{01}^{-1}q_{11}^{-1})\frac{-\lambda_1+\lambda_2+\lambda_3}{2\lambda_1}\right)}.
\end{multline*}
\end{cor}
The formulae in Corollary  \ref{cor: pure theory orbifold} independently appeared in \cite[Prop. 4.26, Prop. 5.20]{ST2} as the partition functions of the pure gauge theories on the quotient stacks.

Applying the insertion-free limit we obtain a \emph{crepant resolution correspondence} for insertion free invariants. 
\begin{cor}
Assume Conjecture \ref{conj: crepant resolution for CY4} holds. Then there exist orientations such that 
\begin{align*}
    \mathcal{Z}^{\DT,\mathrm{free}}_{\xX,\oO_{\xX}}(q_0,\ldots,q_{|G|-1})=\mathcal{Z}^{\DT,\mathrm{free}}_{X,\oO_X}(0,q)\cdot \mathcal{Z}^{\PT,\mathrm{free}}_{X,\oO_X}(Q,q)\cdot \mathcal{Z}^{\PT,\mathrm{free}}_{X,\oO_X}(Q^{-1},q),
\end{align*}
under the change of variables $Q^{\beta_i}=q_i$, $q=q_0\cdots q_r$, where $\beta_i$ is the curve class corresponding to the irreducible $G$-representation $R_i$ by Reid's generalized McKay correspondence \eqref{equ on mckay}.
\end{cor}

\end{document}